\documentclass[]{article}

\usepackage{graphicx}
\usepackage{dcolumn}
\usepackage{bm}
\usepackage{amsmath}
\usepackage{amssymb}
\usepackage[margin=2cm]{geometry}
\usepackage{hyperref}
\usepackage{pgf,pgfarrows,pgfnodes,pgfautomata,pgfheaps,pgfshade}
\usepackage{graphicx,amsfonts}
\usepackage{amsmath,amssymb,setspace}
\usepackage{mathrsfs}
\usepackage{colortbl}
\usepackage{tikz}
\usepackage{framed}
\usepackage{subfig}
\usepackage{enumerate}
\usepackage{epstopdf}

\usepackage{color}                           

\usepackage{geometry}
\usepackage{movie15}
\usepackage{hyperref}

\hypersetup{
    colorlinks,
    urlcolor=blue
}

\numberwithin{equation}{section}
\newtheorem{thm}{Theorem}[section]

\newtheorem{theorem*}{Theorem}
\newtheorem{theorem}[thm]{Theorem}
\newtheorem{remark}[thm]{Remark}

\newtheorem{definition}[thm]{Definition}

\newenvironment{proof}{{\bf Proof.}}{\hfill$\square$\vskip.5cm}

\usepackage[T1]{fontenc}
\usepackage[utf8]{inputenc}
\usepackage{authblk}

\title{
 Differential Geometry Perspective of Shape Coherence and Curvature Evolution by Finite-Time Nonhyperbolic Splitting}

\author[]{Tian Ma\thanks{Email: \href{mailto:mat@clarkson.edu}{mat@clarkson.edu}} } 

 \author[]{Erik M. Bollt\thanks{Corresponding author. Email: \href{mailto:bolltem@clarkson.edu}{bolltem@clarkson.edu}}}

\affil[]{Department of Mathematics and Computer Science, Clarkson University, USA. \\  Accepted to SIADS}

\begin{document}

\date{\today}

\maketitle
\begin{abstract}
Mixing and coherence are essential topics for understanding and describing transport in fluid dynamics and other non-autonomous dynamical systems.
Only recently, the idea of coherence has become to a more serious footing, and particularly within the recent advances of  finite-time studies of nonautonomous dynamical systems.
Here we define {\it shape coherent sets} as a means to emphasize the intuitive notion of ensembles which ``hold together" for some period of time, and we contrast this notion to other recent perspectives of coherence, notably ``coherent pairs", and likewise also to the geodesic theory of material lines.
We will relate shape coherence to the differential geometry concept of curve congruence through matching curvatures.   We show that points in phase space where there is a zero-splitting between stable and unstable foliations locally correspond to points where curvature will evolve only slowly in time.
Then we develop curves of points with zero-angle, meaning non-hyperbolic splitting, by continuation methods in terms of the implicit function theorem.  From this follows a simple  ODE description of the boundaries of shape coherent sets.  We will illustrate our method with popular benchmark examples, and further investigate the intricate structure of foliations geometry.

\end{abstract}



\section{Introduction}
Finite time mixing and transport mechanisms in two-dimensional fluid flows have been a classic and long standing problem area in dynamical systems. 
We are interested here especially in coherent stuctures in the flow.
For the autonomous case, there are several study on almost invariant sets.  See \cite{DJ, FK,BN}, to name a few.
For the nonautonomous case, 
Lagrangian Coherent Structures (LCS) based on Finite Time Lyapunov Exponents (FTLE) that focus on maximal local stretching has become a popular computational way to study  transport, see  \cite{H1, H2}, but with caveats regarding the possibility of false positives have been revealed \cite{Haller-variationalFTLE}; similarly Finite Sized Lyapunov Exponents (FSLE) show differences and likely false positives, \cite{FSLE}.   On the other hand, transfer operator methods based on Galerkin-Ulam matrices which evolve distributions of ensembles of initial conditions have also been successfully developed for finding so-defined coherent pairs, \cite{FH, FK, FSM}.
There are also some recent works on improving, comparing and generalizing these methods, \cite{TR, MB}.  An exciting recent development involves 
minimally stretching material lines by studying the Cauchy-Green strain tensor \cite{HB}, in a so-called theory of geodesic curves.    In \cite{FH}, the authors developed a theory of geodesic curves whereby they define a principle that coherent sets should correspond to boundaries that are minimally stretching.  Our definition here, of shape coherent sets is agreeable with this geodesic curve principle by Farazmand-Haller, but from a different perspective, since as we will discuss that evolution of curvature is also related to evolution of arc length, whereas they discuss evolution of length of boundary curves.

However, popularly  general notions of coherence broadly interpreted may include allowing for sets that may stretch and fold when considering material curves of set boundaries, and ensembles of set of points therein that advect with the flow.  We remind that if highlighted by coloring (partition) then general sets automatically catch the eye, especially when remembering the simple notion that dynamical systems being continuous will preserve properties such as connectedness.  From this perspective, almost any connected set that is highlighted may appear coherent in the sense of continuity.  We take a narrower view here.

Here we introduce a different but related mathematical definition of coherence by defining  {\it shape coherent sets} which we so named to distinguish it from other possible notions of the phrase coherence.   We will show our intuitively motivated notion of shape coherence directly addresses the idea of sets evolving in such a way so as to ``hold together."  Our perspective is strictly Lagrangian, and is meant to capture the typical idea that a coherent set should ``catch your eye," as often described roughly when suggesting coherence.   We mathematically define here this concept in terms of stating that {\it shape coherence} should correspond (approximately at least) to a region of phase space where over the finite time epoch of consideration, the flow restricted to that region is equivalent to a member of the group of rigid body motions. 

Analysis of our definition of shape coherence from a perspective of differential geometry reveals  that sets which have boundaries that have relatively slowly changing curvature of their boundary curves correspond to shape coherent sets.   From this observation we relate that curves with slowly evolving curvature are related to curves with a zero-splitting property, meaning a shear like scenario that when stable and unstable foliations are parallel, and a curve of such points exists, then these curves evolve curvature relatively slowly.  This geometric interpretation by curvature evolution is agreeable  with the theory of  geodesic curves derived through variational principles in \cite{HB}, from which it follows that such curves follow as solutions of the ODE whose vector field is defined by unstable foliations; the connection to coherence is intuitive that  minimally stretching boundary curves suggest slowly changing shapes.  

We  highlight \cite{Schuck} where an observation in discussing 
a ``Lyapunov diffusivity," as they proposed   a quantity that depends on  particle dispersion in the forward time Lyapunov exponent (future) correction
in terms of angle $\alpha$ between 
 stable and unstable Lyapunov vectors,  $\tilde{\lambda}=\lambda \sin^2 \alpha$; from this they note that when $\alpha=0$ then shear would dominate and they note that shear is consistent with a transport barrier. 
 There have been studies of evolution of curvature in the literature, \cite{Liu,Drummond1,Drummond2,Pope,Ishihara}, focused on strongly chaotic regions or turbulent regions.  In the more recent studies of Thiffeault \cite{Thifault1, Thifault2, Thifault3} relating a distribution of curvature growth rate found long tails and corresponding scaling law describing the folding points, and a small range describing the stretching regions, but with a relatively flat intermediate scale suggestive of ellipticity.  The key difference in these previous works, and the appearance of curvature here, is that the past work was focused on the role of curvature in the description of the chaotic set, whereas we show that curvature can also be used to describe the complement, indicative as an equivalence to shape-coherent sets.  See also \cite{Ouellette1,Ouellette2,Ouellette3} for study of curvature of streamlines as descriptive of the topological aspects of experimental turbulent flows.
 
Since it is well known in differential geometry \cite{Carmo} that curvature directly defines curve congruence, here we link directly to our definition of a shape coherent sets through the Frenet-Serret formulae of curvature \cite{Carmo, Frenet, Serret}, and the so-called Frenet frame.  We present a constructive analysis based on the implicit function theorem that such zero-splitting curves may be shown to exist and they may be numerically constructed by adaptations of standard continuation methods.  We give two examples with the popular benchmarks of transport in the Rossby wave system and the double gyre.  We also make some geometric observations that angles between stable and unstable foliations develop zeros with increasing time epoch in a manner that relates to the prominence of the coherent sets, and also related to the accumulation of stable and unstable manifolds progressively revealing Cantor-like structure.

\section{Shape Coherent Sets}

Here we will describe our perspective of coherence in terms of  our definition of  ``Shape Coherence."  We will compare and contrast this definition to the popular coherent pairs definition \cite{FSM} based on measurable dynamics and transfer operators, and also to the geodesic curve theory \cite{FH}, from which we will realize that there are both coincidences in the outcomes, but efficiencies to be gained in a different perspective leading to methods that are computationally very different but toward the same basic goal.  From a differential geometry perspective, we will draw strong links between shape coherence to evolution of curvature of the boundaries of shape coherent sets.

The goal in developing a definition of coherence in a nonautonomous dynamical system is to find sets that ``hold together," in some sense of that phrase requiring a good mathematical definition.
Suppose, a dynamical system,
\begin{equation}\label{vectorfield}
\dot{z}=G(z,t),
\end{equation}
and for a vector field $G:M\times R \rightarrow R$, for an open subset $M\subset R^2$ and sufficient regularity such that there is a flow, $\Phi(z,t;\tau):M\times R \times R$.
 In an autonomous dynamical system, ``hold together" lead  to almost-invariant sets \cite{FK} which describes weak transitivity \cite{F}.  That is, $A$ is almost invariant if $\Phi(A,0;0)\approx \Phi(A,0;T)$, describes almost invariance, for a time epoch, $t=0:T$.  The  idea of almost invariance breaks down in the nonautonomous setting; to demand that the image of a set must mostly overlap itself in finite time is too strict since sets generally move in time.    If we wish to say that two sets $A_t$ and $A_{t+\tau}$ are {\it coherent pairs,} \cite{FSM}, and reviewed in \cite{BN} then in measure $\mu$ we may demand that,
\begin{equation}
\rho_\mu(A_t,A_{t+\tau}):=\frac{\mu(A_t\cap \Phi(A_{t+\tau},t+\tau;\tau))}{\mu(A_t)}\geq \rho_0
\end{equation}
for some fraction $\rho_0\approx 1$.  Further, in \cite{FSM} they demand that $\mu(A_t)=\mu(A_{t+\tau})$.  But since, by this (not yet completely described here) definition with conditions as stated so far here, any set is a coherent pair to its own image, {\it no matter how contorted}.  In \cite{FSM} they further add a third condition to the definition of coherent pairs that $\mu(A_t)$ and $\mu(A_{t+\tau})$ must be ``robust" to small perturbations, while not explicitly part of the computation.  This third condition is covered automatically, in some sense, by the computational diffusive effect \cite{FK1} of the discretized grid used to approximate transfer operators in the Ulam-Galerkin based method \cite{BN}.  By notation here, we will take $A_t$ to be the evolution of an initial set $A$ to time $t$ under the flow, so $A_{t_0}\equiv A$, and for simplicity, we write $\Phi_t(z):=\Phi(z,t_0,t)$ with $t_0=0$.

Alternatively by our perspective we will appeal  directly to the geometry of the idea that a coherent set should ``hold together" to capture the intuition that such sets should ``capture our eyes" when viewing the flow:
%
\begin{definition}\label{scoh1}{\bf Finite Time Shape Coherence}
The shape coherence factor $\alpha$ between two measurable nonempty sets $A$ and $B$ 
under an area preserving flow $\Phi_t$ after a finite time epoch $t\in [0, T]$ is,
{\begin{eqnarray}\label{shapecoherenced}
\displaystyle \alpha(A, B,T) := \sup_{S(B)} \frac{m( S(B) \cap \Phi_{ T}(A)) }{m(B)}
\end{eqnarray}}
where $S(B)$ is a group of transformations of rigid body motions of $B$, specifically translations and  rotations descriptive of {\it frame invariance.}  We interpret $m(\cdot)$ to denote Lebesgue measure, but one may substitute other measures as desired.
Then we say $A$ is finite time shape coherent to $B$ with the factor $\alpha$ under the flow $\Phi_T$ after the time epoch $T$, but stated brieflyT that $A$ is shape coherent to $B$.  We call $B$ as the {\it reference set}, and $A$ as the {\it dynamic set}.
\end{definition}

 Where as generally a nonlinear flow can make for quite complicated distortions, stretching and folding on general sets, on a shape coherent set this definition could be interpreted by the statement that when the flow $\Phi_T$ is restricted to a shape-coherent set $A$, then $\Phi_T|_A$ is roughly equivalent to a simpler transformation, a member of the group $S$ of rigid body transformations.   That is a degree of simplicity on the spatial scale of the set $A$, and on the time scale $T$, even if on finer scales within the set $A$ there may be complicated flow.

\begin{figure*}
  \centering
  \subfloat[A set $A$ in a flow $\Phi$.]{\label{fig:gull}\includegraphics[width=.9\textwidth]{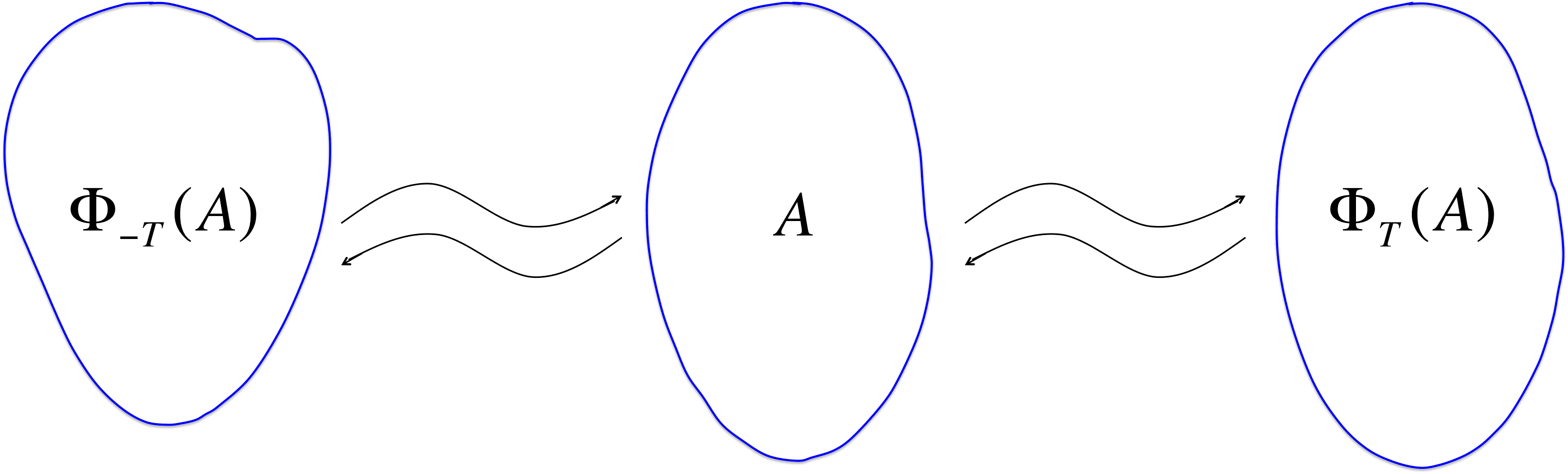}}  \\
  \subfloat[A registration between $\Phi_{-T}(A)$ and $A$]{\label{fig:gull}\includegraphics[width=.5\textwidth]{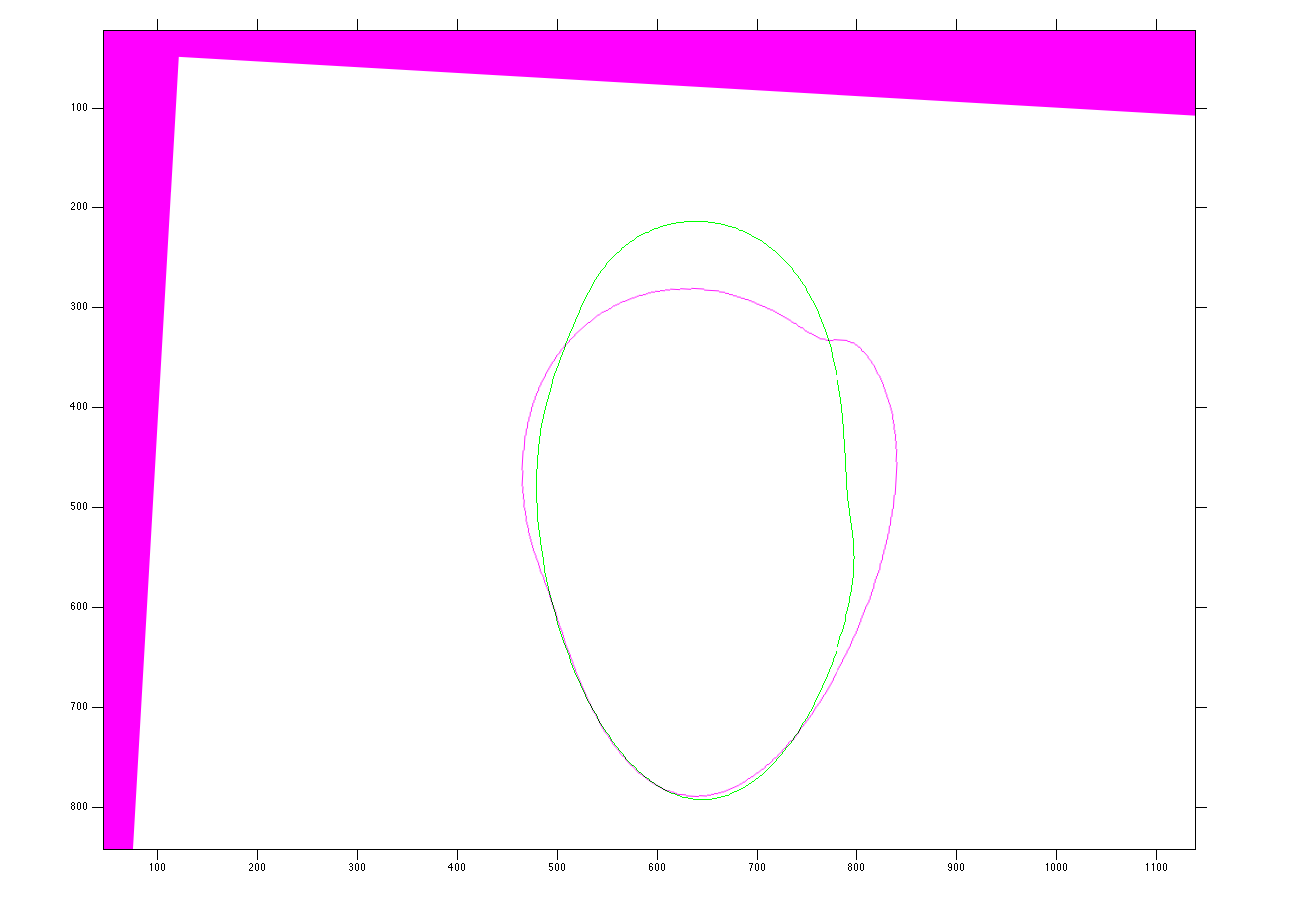}}      
  \subfloat[A registration between $\Phi_{T}(A)$ and $A$]{\label{fig:tiger}\includegraphics[width=.5\textwidth]{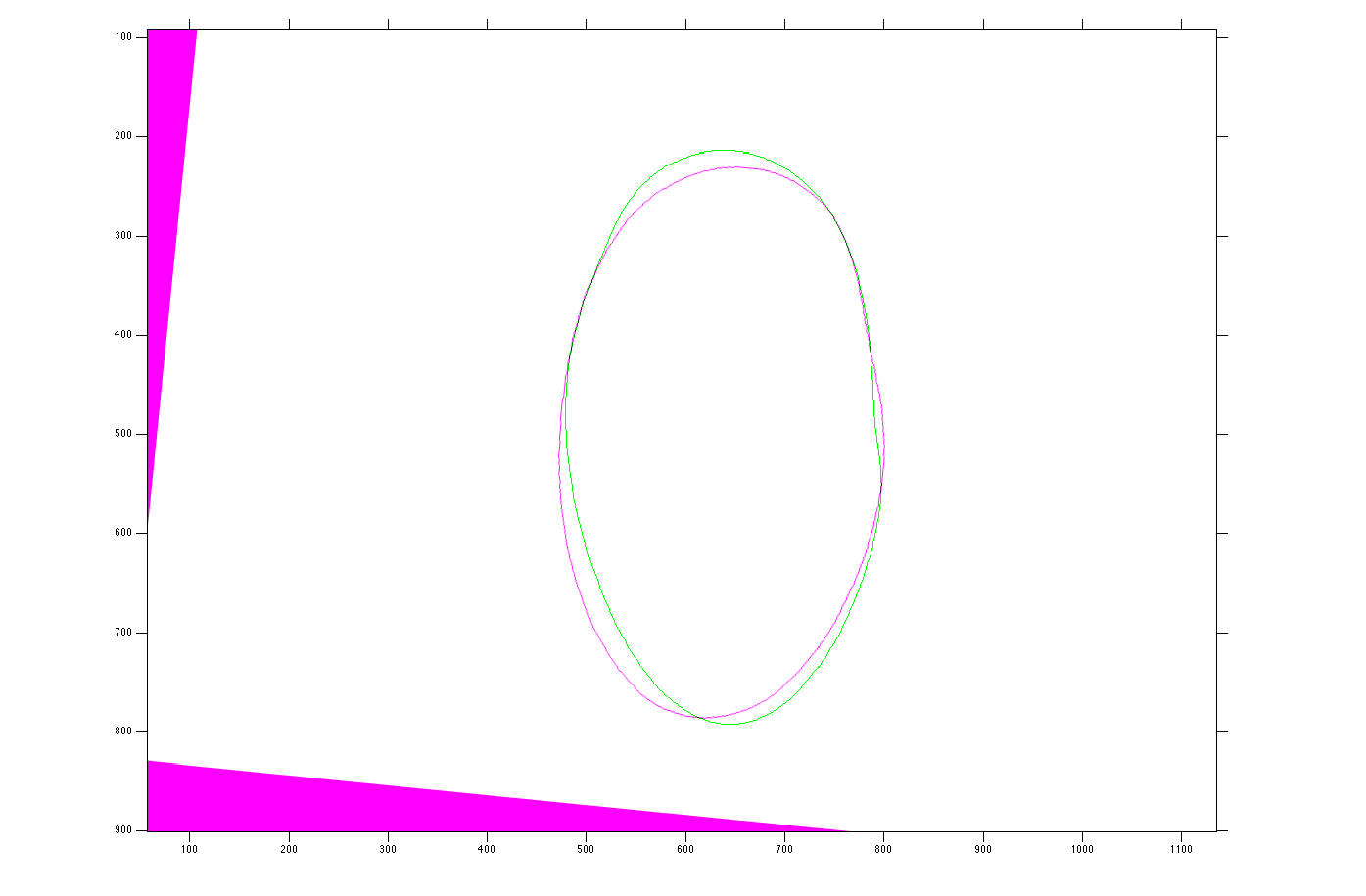}} 
  \caption{(b) and (c) show that $A$ almost keeps its shape in the flow, for both forward time and backward time. Here shapes were matched by direct application of Eq.~(\ref{shapecoherenced}), but specifically by traditional image processing registration algorithms.}
  \label{Registration}
\end{figure*}

\begin{remark}
\begin{enumerate}
\item 
While generally we can choose any kind of set $B$ as the reference set, the idea was designed so that a set corresponding to the intuitive idea of ``holding together" is more relevant.  So for example, one may choose a circle of the same area as $A$ as a reference set.  
\item It is interesting to choose the reference set $B$ to be the dynamic set $A$ itself.  Then we are considering if a set is  {\it finite time shape coherent relative to itself }over time.  That is, {\bf does $A$ maintain its shape?}
%
%
In such case, we may choose $B=A$ at $t=t_0=0$ which is then compared to $\Phi_T(A)$ where $T=t_0+T$.  We are interested here with time epochs centered on $t_0$ and as such, we may choose time intervals $t_1<t_0<t_2$, perhaps to balance stretching. Generally in this paper we will discuss simply centered time, $t_1=-T<t_0=0<t_2=T$.
\item Direct computation of $\alpha$ for specific shapes is called the ``registration" problem within the computational geometry and image processing community, and there are fast algorithms, including based on the FFT to be used for these, \cite{CM, BS, BN}.
\item The definition yields that, $0 \leq \alpha(A,B,T)\leq 1$.
\end{enumerate}
\end{remark}
Any group of transformations could be chosen, but we have restricted here to rigid body motions simply to match to the notion of curvature congruence for the boundary curve in the subsequent section.  However a shape coherence could be defined to also include such transformations as dilation, shear, or even a nonlinear transformation such as a conformal mapping, whether or not it has a group structure.   From the perspective of image warping, registration, and image processing, the group of transformations $S$ serves a role to describe the local transformation $\Phi_T$ {\it as restricted to a chosen set } $A$ and has the characteristics of the simpler transformation from $S$.  In this sense we are asking if $B$ ``warps" to $A$, $\Phi_T(A) \approx S(B)$ as a transformation.
Note that the form of the Definition, Eq.~(\ref{shapecoherenced}), allows a set which may have ``mixing" in its interior $A$ but nonetheless ``holds together" relative to $B$ as a whole. This comes from the specific form of the shape coherence measure, $\alpha$.

We may wish to contrast  a stronger condition that a set remains coherent constantly throughout the time epoch, rather than possibly just at the terminal times as defined by the coefficient $\alpha$ in Defn.~\ref{scoh1}.  In such case we may develop a coefficient $\beta$,
\begin{definition} {\bf Finite Time Shape Coherence Throughout the Time Epoch}
The shape coherence factor throughout the time epoch between between two measurable nonempty sets $A$ and $B$ 
under a flow $\Phi_t$ {\it throughout a complete }finite time epoch $[t_1,t_2]$ is defined as,
{\begin{eqnarray}
\displaystyle \beta(A, B,[t_1,t_2]) := \max_{t\in{[t_1,t_2]}} \displaystyle \alpha(A, B,t).
\end{eqnarray}}
\end{definition}

While $\alpha$ and $\beta$ are well defined, and therefore shape coherence is well defined when $\alpha\approx 1$, or throughout the time epoch if $\beta\approx 1$, direct application may resort to image registration methods, and this may seem awkward.  See in Fig.~\ref{Registration} where shapes $\Phi_{-T}(A)$ and $A$ as well as $A$ and $\Phi_{T}(A)$ were matched by ``traditional" image registration algorithms in Matlab, namely the {\it imregister} algorithm.
Instead we will now describe how these can be easily related to quantities from differential geometry that are easily computed and analyzed, namely the evolution of curvature of the boundaries.

%

\section{Curvature, and Frame Invariance of Shape Coherence}\label{CurveSec}

Here we  relate our definition of shape coherence to the readily computed quantities of the evolution of curvature  of the boundary curves from differential geometry.  In brief, by  Definition \ref{scoh1} above, a set is  shape coherent if it ``mostly" maintains its shape in the sense of a rigid body motion.  Here we relate matching the shape of a set $A$ to its image $\Phi_T(A)$ after the flow, to the problem of matching the boundary curves.  From differential geometry, shapes are matched if their boundary curves are matched and curves are matched if the curvatures of the curves are matched.  Furthermore, we show the  regularity statement  that a slowly changing boundary curve implies shape coherence.  Then in subsequent section we will explore the dynamical properties that correspond to slowly evolving boundary curves.  We begin by recalling the following,

\begin{definition}\label{conger}\cite{Carmo} {\bf Congruence} Two space curves are congruent if they differ only by rigid body motions, meaning translation and rotations.
\end{definition}

Congruence is generally accepted as the equivalence relationship between curves, in differential geometry, as it allows the following theorem that matching curvature and torsion leads to congruence.

\begin{theorem}\cite{Carmo}
{\bf Fundamental Theorem of Curve Theory}.
Two space curves $C$ and $\tilde{C}$ are congruent if and only if their corresponding arcs, $C:{\mathbf \gamma}(s)=(x(s),y(s),z(s)), 0\leq s\leq 1$ and $\tilde{C}: \tilde{{\mathbf \gamma}}(s)=(\tilde{x}(s),\tilde{y}(s),\tilde{z}(s)), 0\leq s\leq 1$, both paramaterized by unit arc length $s$, have curvature and torsion that can be matched, in the sense that there exists a ``shift" parameter $a$ such that $\kappa(s)=\tilde{\kappa}(s')$, and $\tau(s)=\tilde{\tau}(s')$ for all $s$, for $s'=mod(s-a,1)$.
\end{theorem}
We refer to a choice of $s'=mod(s-a,1)$ that matches curvatures and torsion as exactly ``lined-up"  parameterizations, as depicted in Fig.~\ref{CurvatureProp}.
\begin{remark}
In this paper, since we will specialize to planar flows and curves outlining planar sets, here, torsion is  zero, $\tau=\tilde{\tau}=0$, and $z(s)=\tilde{z}(s)=0$.  So we  write simply, ${\mathbf \gamma}(s)=(x(s),y(s))$ and likewise for $\tilde{{\mathbf \gamma}}(s)$.
\end{remark}

Now recalling the Frenet-Serret formula, \cite{Carmo, Frenet, Serret},
\begin{equation}
\left[
\begin{array}{c}
{\mathbf T} \\ {\mathbf N }\\ {\mathbf B}
\end{array}
\right]'=
\left[
\begin{array}{ccc}
0 &  \kappa & 0 \\
-\kappa &  0 &  \tau \\
0 &  -\tau &  0 
\end{array}
\right]
\left[
\begin{array}{c}
{\mathbf T} \\ {\mathbf N} \\ {\mathbf B}
\end{array}
\right],
\end{equation}
but in our planar setting, we get the specialized case,
\begin{equation}\label{FS}
{\mathbf T}'=\kappa{\mathbf  N}, {\mathbf N }'=-\kappa {\mathbf T}.
\end{equation}
The general solution of Eq. (\ref{FS}) can be readily found,
\begin{eqnarray}
{\mathbf T}=C_1 \sin \int_s \kappa (\sigma) d \sigma - C_2 \cos \int_s \kappa (\sigma) d \sigma, \nonumber   \\
{\mathbf N}=C_1 \cos \int_s \kappa (\sigma) d \sigma + C_2 \sin \int_s \kappa (\sigma) d \sigma, 
\label{GeneralSolution}
\end{eqnarray}
where $C_1$ and $C_2$ are constant vectors.

This specifies the tangent and normal vectors of a curve, $T(s)$ and $N(s)$ in terms of a given curvature function, $\kappa(s)$ as an ODE in terms of the derivative of the  (unit) parameterization, $':=\frac{d}{ds}$.  Furthermore, since $T(s)={\mathbf \gamma}'(s)$, solving Eq.~({\ref{FS}) for $T(s)$ (and simultaneously $N(s)$) for a given curvature $\kappa(s)$, allows the curve to be found,
\begin{equation}\label{CurveEqn}
{\mathbf \gamma}(s)=\int_0^s {\mathbf T}(\sigma)d\sigma + {\mathbf \gamma}_0.
\end{equation}
So we see that the curvature specifies the curve, up to an initial position and orientation, ${\mathbf \gamma}_0$.  This clarifies part of the interest in the Frenet frame for the definition of congruence in matching curves.  We wish to point out that the Frenet frame is closely related to the dynamical notion of frame invariance.

Now in the following two theorems, we show  that close curvature functions correspond to close curves.  Then subsequently we show that it then follows that the shape coherence must be significant. The proofs of these two theorems are given in Appendix \ref{appendixproofs}.

\begin{theorem}\label{prop1}
 Given two curvature functions, $\kappa_1(s)$ and $\kappa_2(s)$ of two closed curves $\gamma_1(s)$ and $\gamma_2(s)$ with the same arc length, if $\sup_{s}|\kappa_1(s)-\kappa_2(s)| < \varepsilon$,  
then there exists, 
 \begin{equation}
\delta( \varepsilon)= s^2 \varepsilon ( \|  C_1 \|_2 + \|  C_2 \|_2 ) >0,
\end{equation}
where $C_1$ and $C_2$ are the constant vectors in Eq.~(\ref{GeneralSolution}) and $s$ is the arc length from initial point of the two lined-up curves,
such that  
\begin{equation}
\|\gamma_1(s)-\gamma_2(s) \|_2<\delta(\epsilon).
\end{equation}

\end{theorem}

\begin{figure*}
  \centering
  \subfloat[Two curves with similar curvatures and same arc length.]{\label{fig:gull}\includegraphics[width=.55\textwidth]{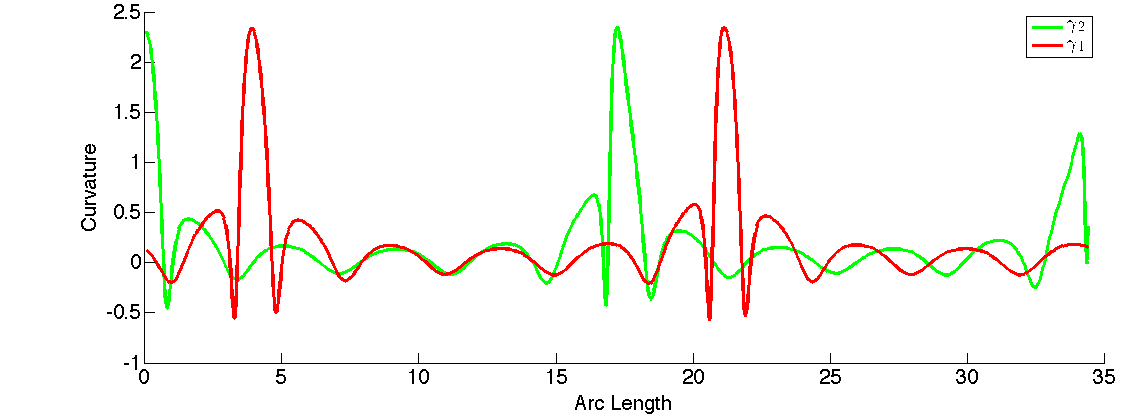}} 
  \centering
  \subfloat[Two curves with similar curvatures and same arc length.]{\label{fig:gull}\includegraphics[width=.55\textwidth]{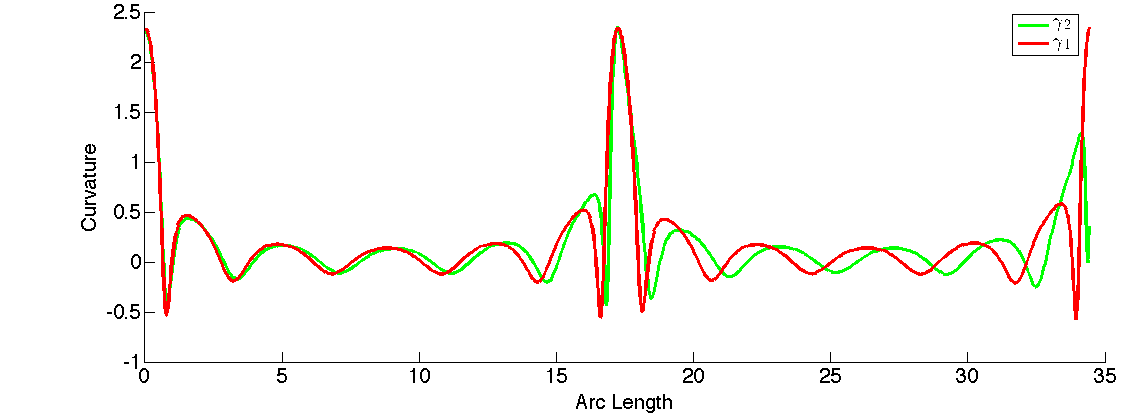}} 
  \qquad
  \qquad
\subfloat[Two curves that have been rectified (lined-up).]{\label{fig:gull}\includegraphics[width=.5\textwidth]{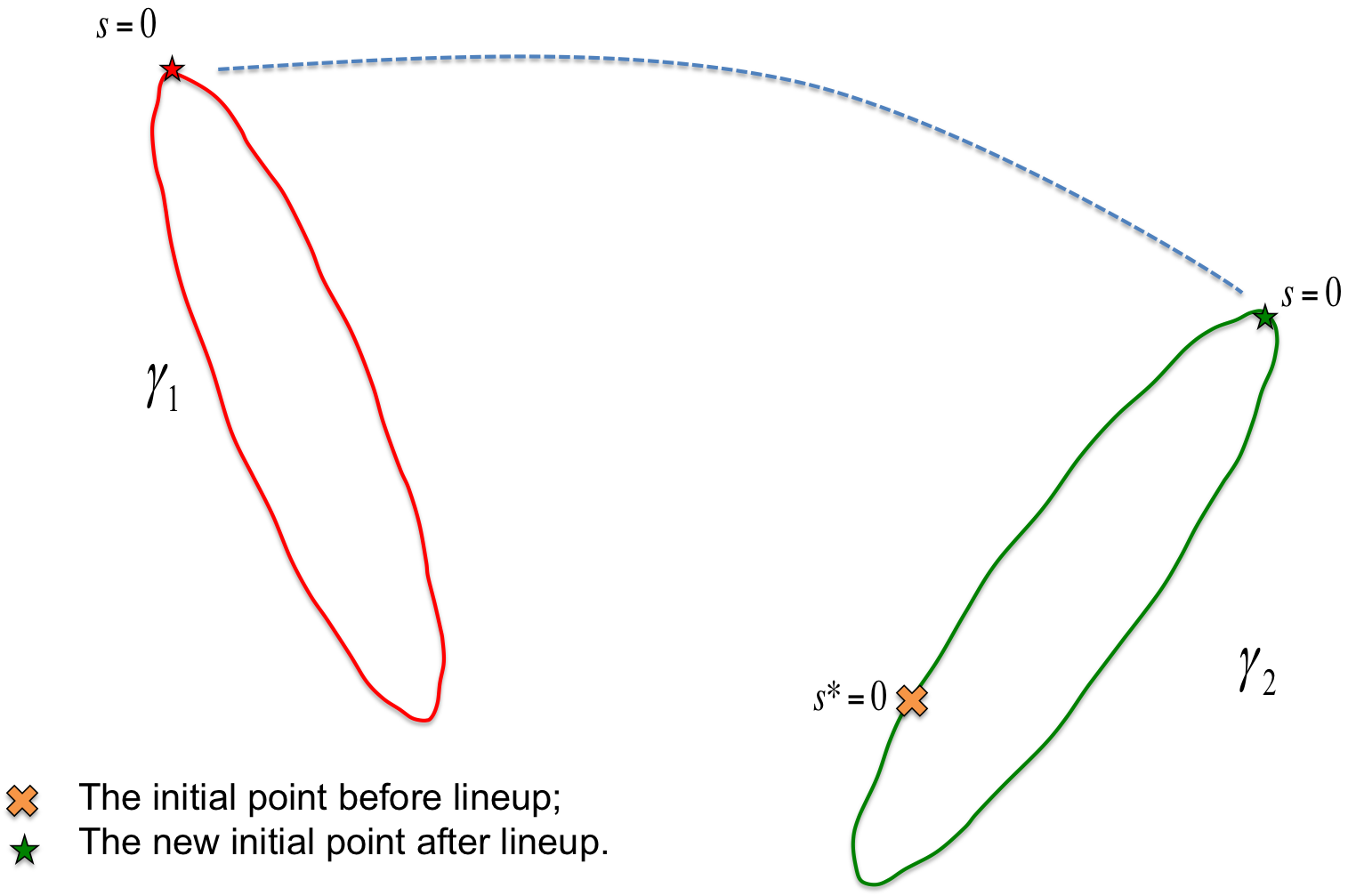}} 
  \subfloat[Two curves that have been rectified (lined-up).]{\label{fig:gull}\includegraphics[width=.5\textwidth]{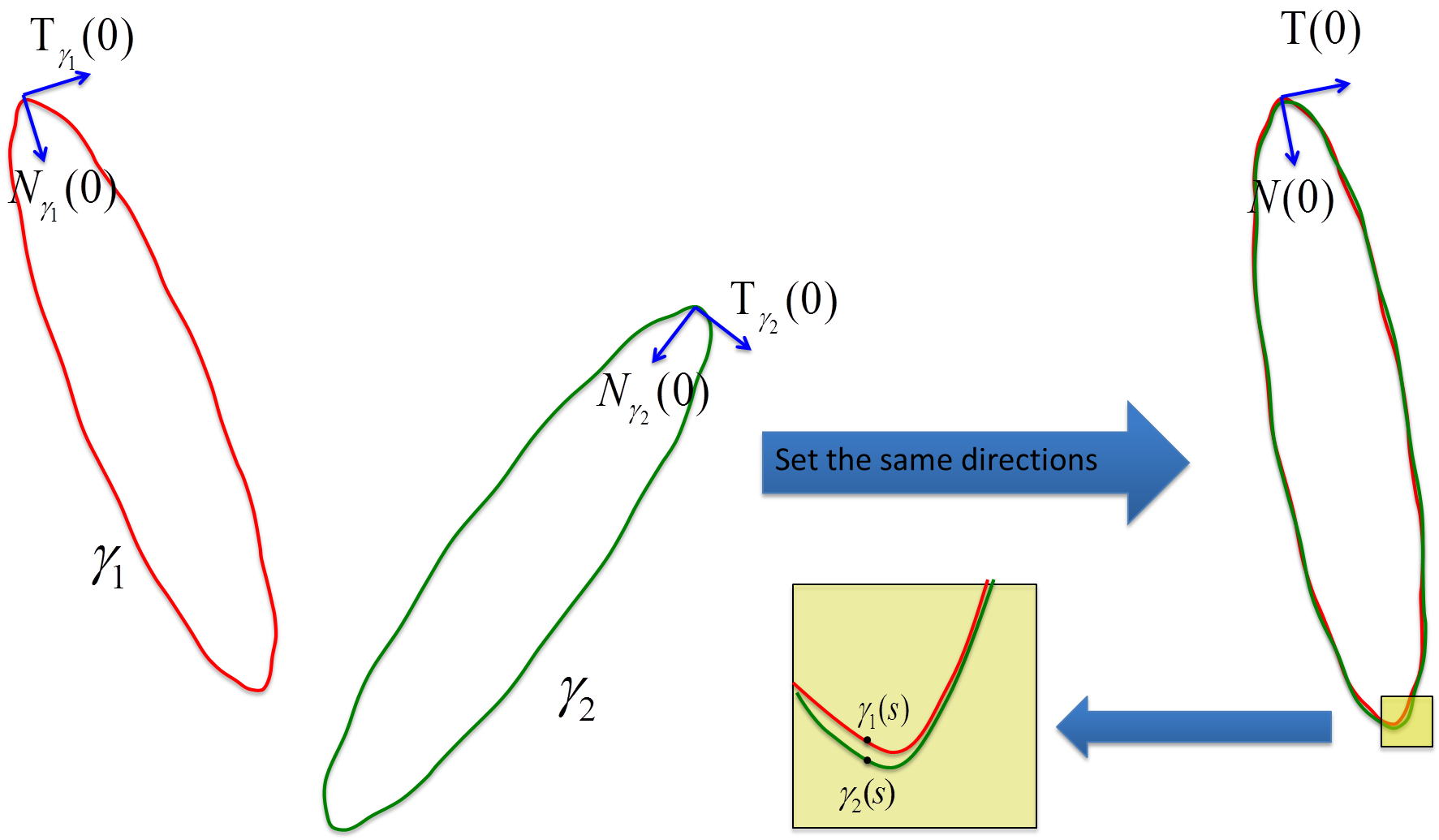}} 
 \caption(c) Two curves with the same arc length are shown, but different the curvatures differ pointwise; note in particular the initial points $s=0$ (red star on $\gamma_1$) and $s^*=0$ (orange cross on $\gamma_2$). In (a) and (b), we apply a circular convolution on the two curves to find a new initial point $s=0$ (green star) on $\gamma_2$ so that it has pointwise almost the same curvature function as $\gamma_1$, as described in Eqs. (\ref{linedup1})-(\ref{linedup2}). (d) With the new initial point on $\gamma_2$, the curves are approximately lined-up  by overlapping and setting the same directions for both the unit normal and unit tangent vectors of the two initial points. 
  \label{CurvatureProp}
\end{figure*}
}

The idea of congruence in Definition \ref{conger} is that comparison of two curves should be immune to  details such as the parameterization of one curve may be shifted relative to the other  despite if they are otherwise (mostly) the same. 
Even if comparing two very similar curves, but which are not oriented in a convenient way, it is possible that $\kappa_1(s_1)$ and $\kappa_2(s_2)$ may not be closely matched pointwise; nonetheless, they may define curves $\gamma_1(s)$ and $\gamma_2(s)$ that are close or even the same in the sense of congruence.  In such case, the only problem to see that curves are so similar by comparing the curvatures is that the parameterizations of each of the two curves  are not ideally ``aligned".  Such a situation is as depicted in Fig.  \ref{CurvatureProp}.  In other words there may exist some shift $a$, such that if we can define,
\begin{equation}\label{linedup1}
s=s_1, \mbox{ and, } s_2=s+a,
\mbox{ such that, }
\sup_{s}|\kappa_1(s)-\kappa_2(s)| < \varepsilon,
\end{equation}
and if it turns out that  $\varepsilon\geq 0$ is sufficiently small, then by Theorem \ref{prop1}, it can be shown that $\gamma_1(s)$ and $\gamma_2(s)$ are close.  
In practice, before 
 resorting to Theorem \ref{prop1}, and assuming continuous functions, it is useful to choose, 
\begin{equation}\label{linedup2}
a=\mbox{argmin} \max_s |\kappa_1(s)-\kappa_2(s)|, \mbox{ where, }s=s_1, \mbox{ and, } s_2=s+a.
\end{equation}
  The curves are approximately lined-up. One useful way we find to estimate such a shift is by  convolution.

In the next theorem we describe how closely matching curves leads to closely matched shapes as measured in terms of shape coherence.  This together with the above that closely matching curvature functions leads to closely matched curves, informs us in the following sections that studying propensity of slowly changing curvature is the key to understanding shape coherence.  To most usefully apply the following theorem, we should again assume generally  that the parameterizations are  designed to associate points along the curves in an efficient way as described in the previous paragraph; otherwise there may exist a parameterization that causes a small $\epsilon$, but we may not realize it without considering Eq.~(\ref{linedup2}).    

\begin{figure*}
  \centering
  \subfloat{\label{fig:gull}\includegraphics[width=.5\textwidth]{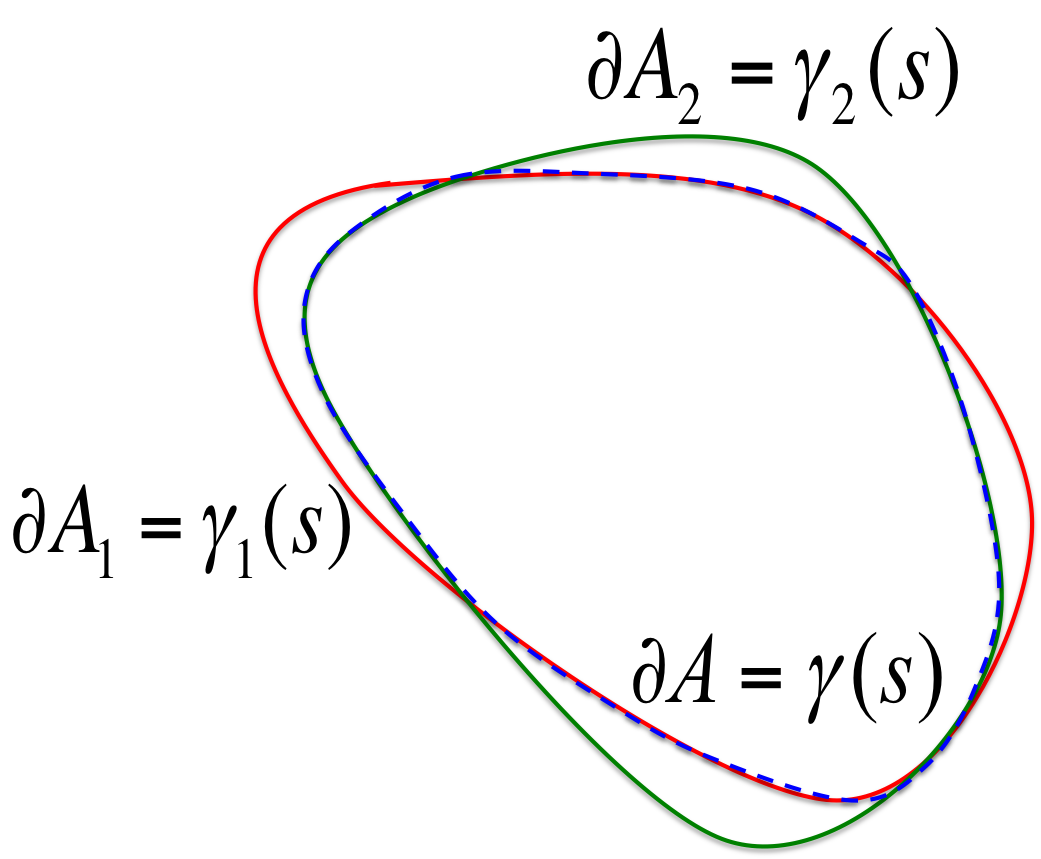}} 
  \caption{The red, green and blue closed curves are the boundaries of sets $A_1, A_2$ and $A$.}
  \label{AreaIntersection}
\end{figure*}

\begin{theorem}\label{3.6}
For two closed curves $\gamma_1(s)=(x_1(s), y_1(s))$ and $\gamma_2(s)=(x_2(s), y_2(s)) (0\le s < 2\pi)$ which are boundaries of sets $A_1$ and $A_2$.   See Fig.~\ref{AreaIntersection}.  Let the boundaries of $A=A_1 \cap A_2$
$A=A_1 \cap A_2\neq \emptyset$ 
such that $area(A)>0$ 
be $\gamma(s)=(x(s),y(s))$, and,
\begin{eqnarray}\label{1stShapeCoherent}
\varepsilon= \max\{|x-x_2|,|y-y_2|,|x'-x_2'|,|y'-y_2'| \} \nonumber \\
M= 2\max\{ |x|,|x_2|,|y|,|y_2|,|x'|,|x_2'|,|y'|,|y_2'|\},
\end{eqnarray}
then there exist a $\Delta(\epsilon)$, which is defined as, 
\begin{eqnarray}\label{2stShapeCoherent}
\Delta(\epsilon)=\frac{2\pi M \varepsilon }{Area(A_2)}
\end{eqnarray}
such that 
\begin{eqnarray}\label{3stShapeCoherent}
1\geq \alpha(A_1, A_2, 0)\ge1-\Delta(\epsilon)
\end{eqnarray}
\end{theorem}


\section{Finite-Time Stable and Unstable Foliations}\label{SMFsection}


Stated simply, the stable foliation at a point describes the dominant direction of local contraction in forward time, and the unstable foliation describes the dominant direction of contraction in ``backward" time.    See Figs.~\ref{SVDVis}-\ref{FoliationAndCurves1}.
Generally, the Jacobian matrix,
$D\Phi_t(z)$ of the flow $\Phi_t(\cdot)$ evaluated at the point $z$ has the same action as does any matrix in that a circle maps onto an ellipse. 
In Figs.~\ref{SVDVis}-\ref{FoliationAndCurves1} we illustrate the general infinitesimal geometry of a small disc of variations $\epsilon w$ from a base point $\Phi_t(z)$.  At $z$, we observe that a circle of such vectors, $w=<\cos(\theta),\sin(\theta)>, 0\leq\theta \leq 2\pi$ centered at the point $\Phi_t(z)$ pulls back under $D\Phi_{-t}({\Phi_{t}(z)})$ to an ellipsoid centered on $z$.  The major axis of that infinitesimal ellipsoid defines  $f_s^t(z)$, the stable foliation at $z$.  Likewise, from $\Phi_{-t}(z)$, a small disc of variations pushes forward under $D\Phi_{t}({\Phi_{-t}(z)})$ to an ellipsoid, again centered on $z$.  The major axis of this ellipsoid defines the unstable foliations, $f_u^t(z)$.

To compute the major axis of ellipsoids corresponding to how discs evolve under the action of matrices, we may resort to the singular value decomposition 
\cite{GV89}.  Let,
\begin{equation}\label{svd}
D\Phi_{t}({z})=U\Sigma V^*, 
\end{equation}
where $^*$ denotes the transpose of a matrix,  $U$ and $V$ are orthogonal matrices, and $\Sigma=diag(\sigma_1,\sigma_2)$ is a diagonal matrix.  By convention we choose the index to order, $\sigma_1\geq \sigma_2\geq 0$.
As part of the standard singular value decomposition theory, principal component analysis provides that the first unit column vector of $V=[v_1|v_2]$ corresponding to the largest singular value, $\sigma_1$, is the major axis of the image of a circle under the matrix $D\Phi_{t}({z})$ around $z$.   That is, 
\begin{equation}
D \Phi_{t}({z}) v_1=\sigma_1 u_1,
\end{equation}
as seen in Fig.~\ref{SVDVis}
 describes the vector $v_1$ at $z$ that maps onto the major axis, $\sigma u_1$ at $\Phi_t(z)$.  Since $\Phi_{-t}\circ \Phi_t(z)=z$, and $D\Phi_{-t}(\Phi_t(z)) D \Phi_t(z)=I$, then recalling the orthogonality of $U$ and $V$, it can be shown that,
\begin{equation}
D\Phi_{-t}({\Phi_t(z)})=V\Sigma^{-1} U^*, 
\end{equation}
and $\Sigma^{-1}=diag(\frac{1}{\sigma_1},\frac{1}{\sigma_2})$.  Therefore, $\frac{1}{\sigma_2}\geq \frac{1}{\sigma_1}$, and the dominant axis of the image of an infinitesimal circle from  $\Phi_t(z)$ comes from, 
$D \Phi_{t}({z}) u_2=\frac{1}{\sigma_2} v_2$.  

\begin{figure*}
  \centering
  \subfloat{\label{fig:gull}\includegraphics[width=.8\textwidth]{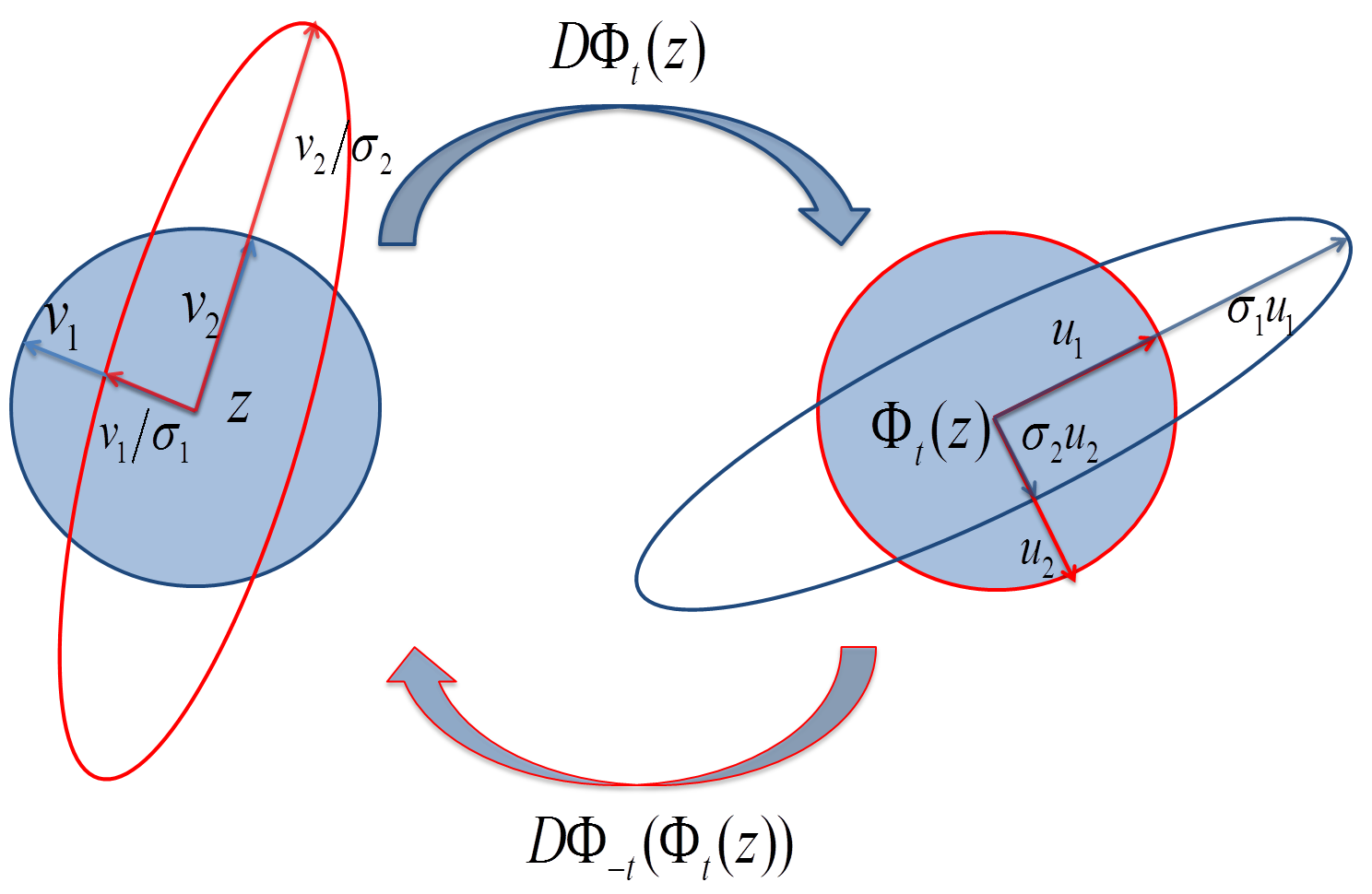} }
  \caption{The SVD Eq.~(\ref{svd}) of the flow $\Phi_t(z)$ can be used to infer the finite time stable foliation $f_s^t(z)$ (and likewise finite time unstable foliation $f_u^t(z)$ at $z$ in terms of the major and minor axis as shown and described in Eqs.~(\ref{svdfs})-(\ref{svdfu}).  }
  \label{SVDVis}
\end{figure*}

We summarize, the stable foliation at $z$ is,
\begin{equation}
 f_s^t(z)=v_2, 
 \end{equation}
where $v_2$ is the second right singular vector of $D\Phi_{t}({z})$, according to Eq.~(\ref{svd}).  Likewise by description above, the unstable foliation,
\begin{equation}\label{svdfs}
 f_u^t(z)=\overline{u}_1, 
 \end{equation}
 where $\overline{u}_1$ is the first left singular vector of the matrix decomposition, 
 \begin{equation}\label{svdfu}
 D\Phi_{t}({\Phi_{-t}(z)})=\overline{U} \mbox{ } \overline{\Sigma} \mbox{ } \overline{V}^*.
 \end{equation}

An important concept here is the included angles between the stable and unstable foliations as follows.
\begin{definition}
The included angle of the finite-time stable and unstable foliations is defined as 
$ \theta(z, t): \Omega \times \mathbb{R}^+ \to [-\pi/2, \pi /2]$
{\begin{eqnarray}
\theta(z, t):=\arccos\frac{\left\langle f_s^t(z), f_u^t(z) \right\rangle}{\|f_s^t(z) \|_2 \| f_u^t(z) \|_2}
\label{AngleFunction}
\end{eqnarray}}
\end{definition}
See the included angle indicated in Fig.~\ref{FoliationAndCurves1}, which plays a role in evolution of curvature and shape coherence as discussed in the subsequent section.

\section{Curvature Evolution Near Local Hyperbolicity and Nonhyperbolicty}\label{seccurvesofstableunstable}

In Sec.~\ref{CurveSec} we presented theory that shapes whose boundary curves have the property that their curvature is slowly varying in time correspond to shape coherent sets.
  In this section we will argue on geometric grounds that curves whose points  correspond to tangency between stable and unstable foliations, as presented in Sec.~\ref{SMFsection},  tend to have slowly evolving curvature.  Thus stable and unstable foliations are related to shape coherence.  For this reason, in the next section we will present theory and later a constructive method to find curves of such tangencies as a means to construct shape coherent sets.

Before tackling the general problem of nonlinear flows as evolved in finite time, we discuss the linear flow, representing the evolution of curvature in the neighborhood of a base point $z=(x_1,x_2)$ with each of several hyperbolic an non hyperbolic scenarios.

Here we cover a hyperbolic saddle which we contrast to three (precursor of) non-hyperbolic transformation types, scaling, rotation and shear \cite{B,L}. We recover the curvature formula under each of these linear transformation maps, representing the local behaviors of the flow near a point of zero-splitting foliation.

\begin{enumerate}

\item Hyperbolic Saddle;

Let,
\begin{equation}\label{dec}
\dot{x}_1=\lambda_1 x_1, x_1(0)=x_{1,0}, x_2(0)=x_{2,0}, \dot{x}_2=\lambda_2 x_2, \lambda_2>0>\lambda_1,
\end{equation}
which by design is a hyperbolic saddle and decoupled, so the stable and unstable manifolds are orthogonal.  Now consider the evolution of a curve of initial conditions, 
\begin{equation}
x_{2,0}=f(x_{1,0}).
\end{equation}
As an example of how this curve evolves in such a manner so as to increase curvature at the origin, see Fig.~\ref{FoliationAndCurves1}a.  To verify this statement mathematically,  
let the evolution of points on the curve starting at initial conditions $(x_{1,0},x_{2,0})=(x_{1,0},f(x_{1,0}))$,  evolves according to the linear flow,
\begin{equation}
(x_1(s,t),x_2(s,t))=(s e^{\lambda_1 t}, f(s) e^{\lambda_2 t}),
\end{equation}
taking $s=x_{1,0}$ to be the chosen parameterization of the curve in terms of the initial $x_1$-position of a point on the curve.  See  Fig.~\ref{FoliationAndCurves1}.
Then using the standard curvature computation of a two-dimensional parameterized curve yields,
\begin{equation}
k(s,t)=\frac{|x_1'x_2''-x_2'x_1''|}{(x_1'^2+x_2'^2)^{3/2}}=\frac{e^{(\lambda_2-2 \lambda_1)t}|f''(s)|}{[1+f'(s)^2 e^{2(\lambda_2-\lambda_1)t}]^{3/2}}.
\end{equation}

\begin{figure*}
  \centering            
   \subfloat[]{\label{fig:tiger}\includegraphics[width=1.05 \textwidth]{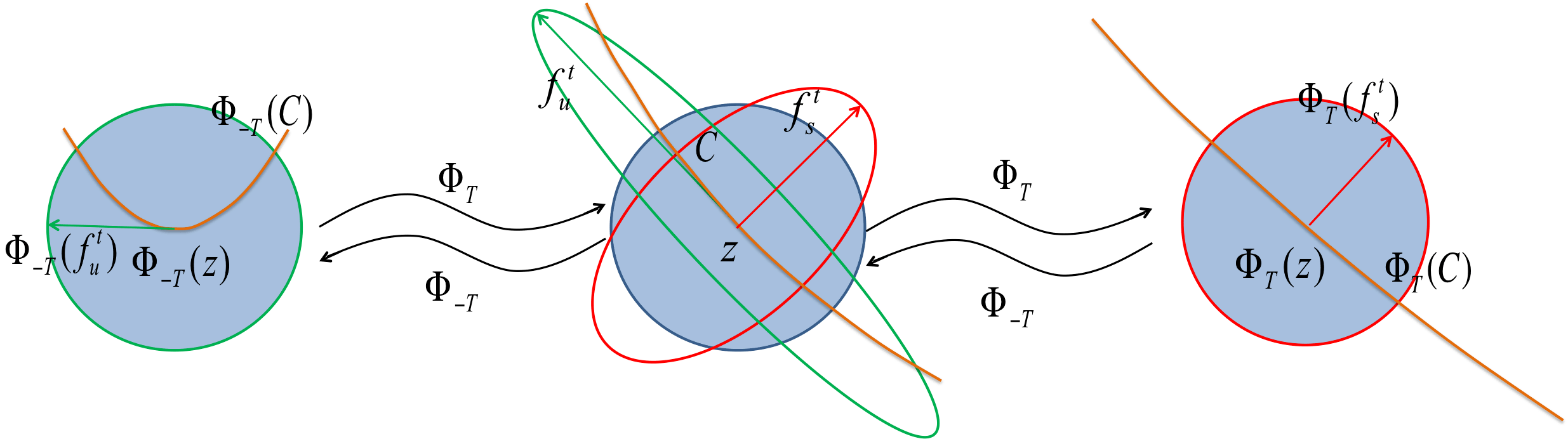}}  \\
   \subfloat[]{\label{fig:tiger}\includegraphics[width=1 \textwidth]{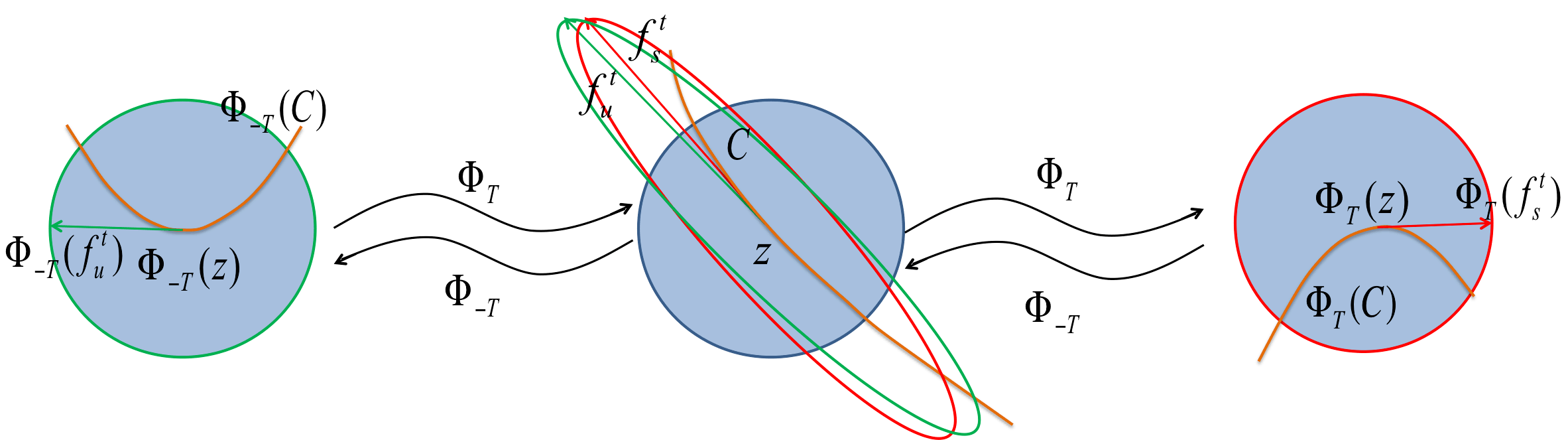}}  
  \caption{ (a) A curve $C$ goes through a small neighborhood of a point $z$ with $90$ degree foliations angle changes its shape from time $-T$ to $T$.  Notice that the curve changes its curvature significantly in time, and it can increase or decrease curvature depending on the details of how the curve is oriented relative to $f^t_s(z)$ and $f^t_u(z)$; (b) The same curve $C$ but with {\it almost} zero-splitting foliations roughly keeps its shape as noted by inspecting the curvature at $z$ through time. }
  \label{FoliationAndCurves1}
\end{figure*}

Thus we may estimate asymptotically in long time,
\begin{equation}
k(s,t)\approx e^{(\lambda_1-2\lambda_2)t}|f''(s)f'(s)^{-3}|, \mbox{when } t>>1
\end{equation}
where long time is interpreted as $t>0$ when $e^{2(\lambda_1-\lambda_2)t}<f'(s)^2$.  Conversely, asymptotically in short time, 
\begin{equation}
k(s,t)\approx e^{(\lambda_2-2\lambda_1)t}|f''(s)|, \mbox{when } t<<1,
\end{equation}
where short time is interpreted as $e^{2(\lambda_2-\lambda_1)t}<f'(s)^{-2}$.
We can interpret that 
since we have assumed that $\lambda_2>0>\lambda_1$, then in short time curvature grows, but after reaching a maximum, in long time curvature shrinks. Notably add to this story that points of large initial curvature due to large $f''(s)$ and small $f'(s)$ persist longer in growing stage of the curvature as noted by $e^{2(\lambda_2-\lambda_1)t}<f'(s)^{-2}$ before transition to  $e^{2(\lambda_1-\lambda_2)t}<f'(s)^2$.  Thus a highly hyperbolic saddle structure suggests significant change of curvature which corresponds to significant changes in shape according to the Frenet-Serret theory.  

\item Scaling and rotation;

For a given time epoch $0<t<+\infty$ and a point $z=(x,y)$, we suppose the forward flow $F_t$ is

\[ F_t (z)= \left( \begin{array}{cc}
 \cos\alpha & -\sin \alpha \\
\sin \alpha & \cos \alpha \end{array} \right)
\left( \begin{array}{cc}
 a & 0 \\
0 & \frac{1}{a} \end{array} \right) z
\] 
and the backward transport matrix $B_t$ is 
\[ B_t (z)= \left( \begin{array}{cc}
 \cos\beta & -\sin \beta \\
\sin \beta & \cos \beta \end{array} \right) 
\left( \begin{array}{cc}
 b & 0 \\
0 &  \frac{1}{b}  \end{array} \right) z
\] 
where $\beta, \alpha>0$ and $a, b>1$, without loss of generality. 
See Fig. \ref{Appendix1} (a). 
Thus, the flow from $-t$ to $t$ can be considered as $B_t^{-1} F_t(z)$. 
And the deformation matrix of $B_t^{-1} F_t(z)$ clearly is the same as its coefficient matrix, which is,

\begin{figure*}
  \centering           
   \subfloat[]{\label{fig:tiger}\includegraphics[width=.9 \textwidth]{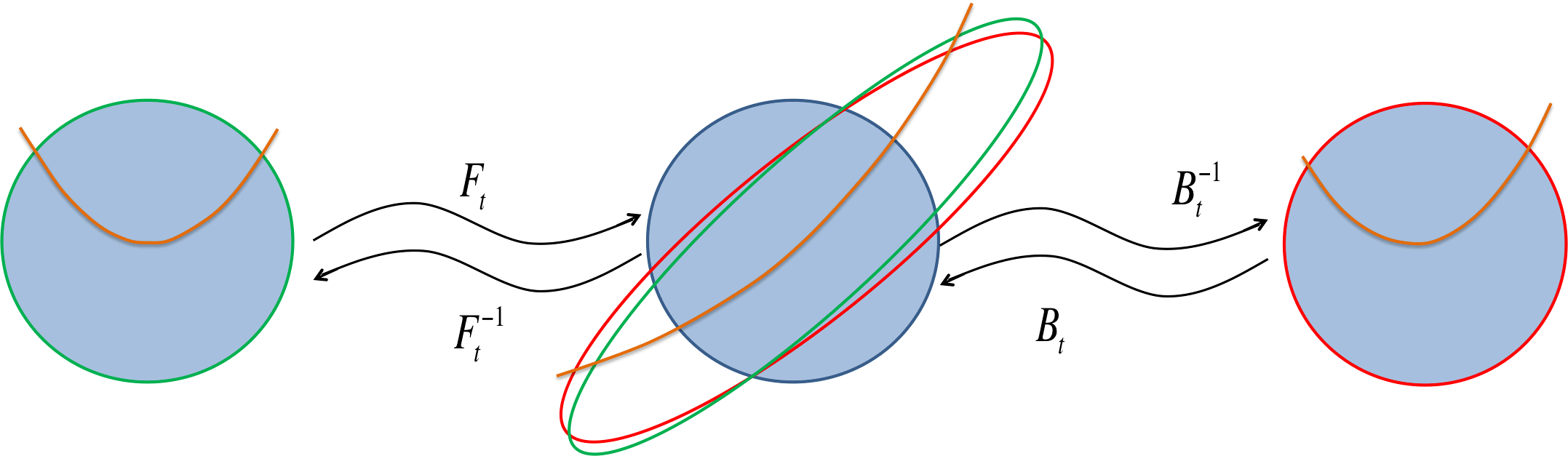}}   \\
   \subfloat[]{\label{fig:tiger}\includegraphics[width=1 \textwidth]{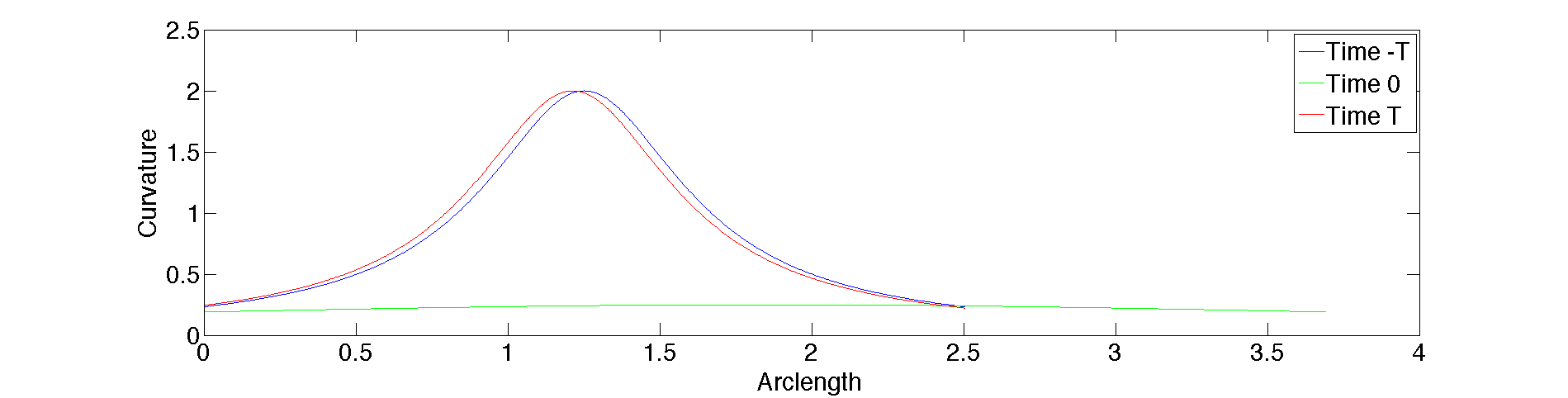}}  
  \caption{The rotation and scaling case. We see that curvature changes only slightly over the time epoch.}
  \label{Appendix1}
\end{figure*}

\[ Q= B_t^{-1} F_t =\left( \begin{array}{cc}
 \frac{1}{b} & 0 \\
0 & b \end{array} \right)
\left( \begin{array}{cc}
 \cos \theta & -\sin \theta \\
\sin \theta & \cos \theta \end{array} \right)
\left( \begin{array}{cc}
 a & 0 \\
0 &  \frac{1}{a}  \end{array} \right)
=\left( \begin{array}{cc}
  \frac{a}{b} \cos \theta & - \frac{1}{ab}  \sin \theta \\
ab \sin \theta &  \frac{b}{a} \cos \theta \end{array} \right)
\] 
where $\theta=mod(\alpha-\beta, \pi/2).$ 
And it is easy to show that $\theta$ is the included angle of the stable and unstable foliations. 
We next show that how the angle $\theta$ related to the change of curvature of a given curve under the flow from $-t$ to $t$.

Consider a curve $\gamma_{-t}=(x_{-t},f(x_{-t}))$ at time $-t$, under the flow through a time interval $[-t, t]$. 
The curvature of $\gamma_{-t}$ at $-t$ is known as, 
\begin{eqnarray}
k(\gamma_{-t})=\frac{|f''(x_{-t})|}{(1+f'^2(x_{-t}))^{\frac{3}{2}}}
\end{eqnarray}
However, the curvature is changed by the flow $B_t^{-1} F_t(z)$ to, 
\begin{eqnarray}
k(\gamma_{t})=\frac{|f''(x_{-t})|}{((\frac{a}{b} \cos \theta - \frac{1}{ab} f'^2(x_{-t}) \sin \theta )^2 
+(ab \sin \theta+\frac{b}{a} f'^2(x_{-t}) \cos \theta)^2)^{\frac{3}{2}}}
\end{eqnarray}
The ratio between the two curvatures can be written,
\begin{eqnarray}
\frac{k(\gamma_{-t})}{k(\gamma_{t})}=\frac{(\frac{a}{b} \cos \theta - \frac{1}{ab} f'^2(x_{-t}) \sin \theta )^2 
+(ab \sin \theta+\frac{b}{a} f'^2(x_{-t}) \cos \theta)^2}{1+f'^2(x_{-t})}.
\end{eqnarray}
The Taylor expension of the ratio $\frac{k(\gamma_{-t})}{k(\gamma_{t})}$ with respect to small angle $\theta$ is, 
\begin{eqnarray}
\frac{k(\gamma_{-t})}{k(\gamma_{t})}=
\frac{\frac{a^2}{b^2}+\frac{b^2f'^2(x_{-t})}{a^2}}{1+f'^2(x_{-t})}+\frac{2f'(x_{-t})(b^2-\frac{1}{b^2})}{1+f'^2(x_{-t})} \theta + O(\theta ^2)
\end{eqnarray}
Thus, for $\theta<<1$, if the flow has $a \approx b$, 
we have $\frac{k(\gamma_{-t})}{k(\gamma_{t})} \approx 1$ for all $x_{-t}$ of the curve. 
See Fig. \ref{Appendix1} (b). 
The curvatures nearly no change through time interval $[-t, t]$.
Note that there may be some special $x_{-t}$ such that this ratio may be close to $1$ even without a small $\theta$. 

\item Shear;

\begin{enumerate}

 \item If we have 
\[ F_t= \left( \begin{array}{cc}
 1 & a \\
0 & 1 \end{array} \right)
\] 
and $B_t$ is 
\[ B_t= \left( \begin{array}{cc}
 1 & b \\
0 & 1 \end{array} \right)
\]
then follows 
\[ Q=B_t^{-1} F_t= \left( \begin{array}{cc}
 1 & -b \\
0 & 1 \end{array} \right)
\left( \begin{array}{cc}
 1 & a \\
0 & 1 \end{array} \right)
\] 

By the same process of case 1, we have the curvatures ratio as,
\begin{eqnarray}
\frac{k(\gamma_{-t})}{k(\gamma_{t})}=\frac{ (1+(a-b)f'(x_{-t}))^2+f'^2(x_{-t})}{1+f'^2(x_{-t})}.
\end{eqnarray}
Thus,  $a\approx b$  is necessary and sufficient for the angle between foliations to be small, 
and from this follows that the the ratio $\frac{k(\gamma_{-t})}{k(\gamma_{t})} \approx 1$. 
See Fig. \ref{Appendix2}.


\begin{figure*}
  \centering           
   \subfloat[]{\label{fig:tiger}\includegraphics[width=.9 \textwidth]{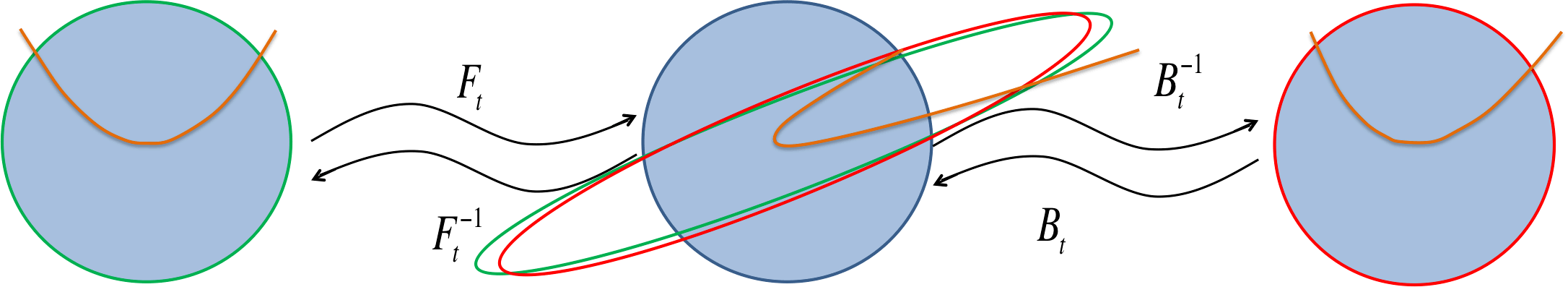}}   \\
   \subfloat[]{\label{fig:tiger}\includegraphics[width=1 \textwidth]{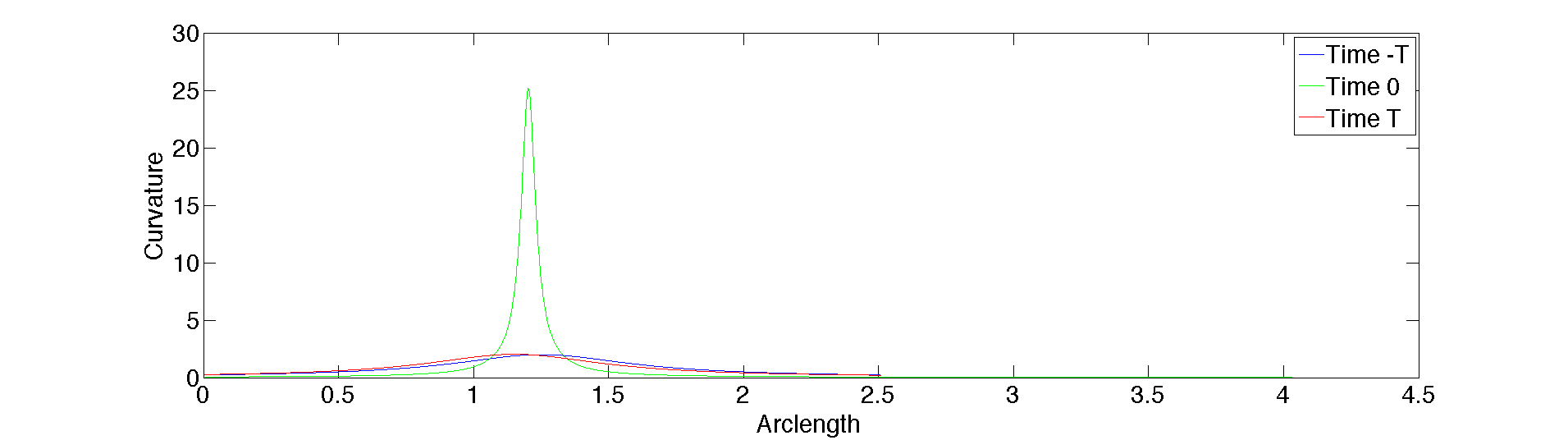}}  
  \caption{The shear case.  Again curvature changes only slightly over the time epoch.}
  \label{Appendix2}
\end{figure*}

%
\item If we have 
\[ F_t= \left( \begin{array}{cc}
 1 & 0 \\
a & 1 \end{array} \right)
\] 
and $B_t$ is 
\[ B_t= \left( \begin{array}{cc}
 1 & b \\
0 & 1 \end{array} \right)
\]. 
then, 
\[ Q=B_t ^{-1} F_t= \left( \begin{array}{cc}
 1 & -b\\
 0 & 1 \end{array} \right)
\left( \begin{array}{cc}
 1 & 0 \\
a & 1 \end{array} \right)
\] 

In this case, we can see the angle between foliations never reaches $0$ unless $a=b=0$. 
However, a small angle can still can keep the curvature relatively constant.
Consider the ratio between curvatures,
\begin{eqnarray}
\frac{k(\gamma_{-t})}{k(\gamma_{t})}=\frac{ (1-ab-bf'(x_{-t}))^2+(a+f'(x_{-t}))^2}{1+f'^2(x_{-t})},
\end{eqnarray}
Hence, if we have $a\approx 0$ and $b\approx 0$, the foliations' angle is small and the ratio $\frac{k(\gamma_{-t})}{k(\gamma_{t})} \approx 1$.

\end{enumerate}

\end{enumerate}

\section{Curves of Stable and Unstable Foliations Tangencies}\label{seccurvesofstableunstable}

Motivated by these linear hyperbolic and non hyperbolic dynamics, we consider a nonlinear flow and how it evolves a complete curve of points $z$, such that each point on the curve locally has a tangency scenario as above.  Therefore we should expect that at each point on such a curve, curvature will change slowly.  We will highlight the difference between  hyperbolicity and nonhyperbolicity in terms of evolution of curves by showing a non hyperbolic curve and contrast with a nearby hyperbolic curve.  A question which follows this discussion is if there is a curve that only consists of  points with a zero-splitting angle, i.e. a zero-splitting curve, is a set of points $C$ such that,
\begin{equation}
\mbox{If }z\in C \mbox{ then } \theta(z,t)=0,
\end{equation}
The answer is yes. 
In this section we will not discuss construction as that will be covered in  Sec.~\ref{continuation1} to follow.

Fig. \ref{FoliationAndCurves4} shows the comparison between a zero-splitting curve and a nearby general curve $\tilde{C}$, 
We show time evolution of $C$  from time $-T$ until time $T$ and correspondingly that its curvature changes only slightly.  However,  even a nearby curve $\tilde{C}$ is shown and significant changes in curvature develops in the same time epoch.  Correspondingly we see that $C$ evidently encloses a shape coherent set, but $\tilde{C}$ does not.  It may seem surprising that a small displacement of 
$C$ produces such a large change of evolution of curvature, but this preludes the answer to this surprise which relates to the geometry that finite time stable and unstable foliations can change direction quite rapidly, even in small neighborhoods, as suggested by Figs.~\ref{RW1}-\ref{RWAngles} in Sec.~\ref{examples}.
This is agreeable with the traditional
 well known infinite time concept of stable manifolds accumulating on unstable manifolds, known as the lambda lemma, \cite{lambdalemma}. 
Thus motivated, in the next section we will discuss a continuation algorithm  based on the implicit function theorem to construct curves $C$ of zero-splitting.

\begin{figure*}
  \centering            
   \subfloat[A zero-splitting curve versus a general curve]{\label{fig:tiger}\includegraphics[width=1 \textwidth]{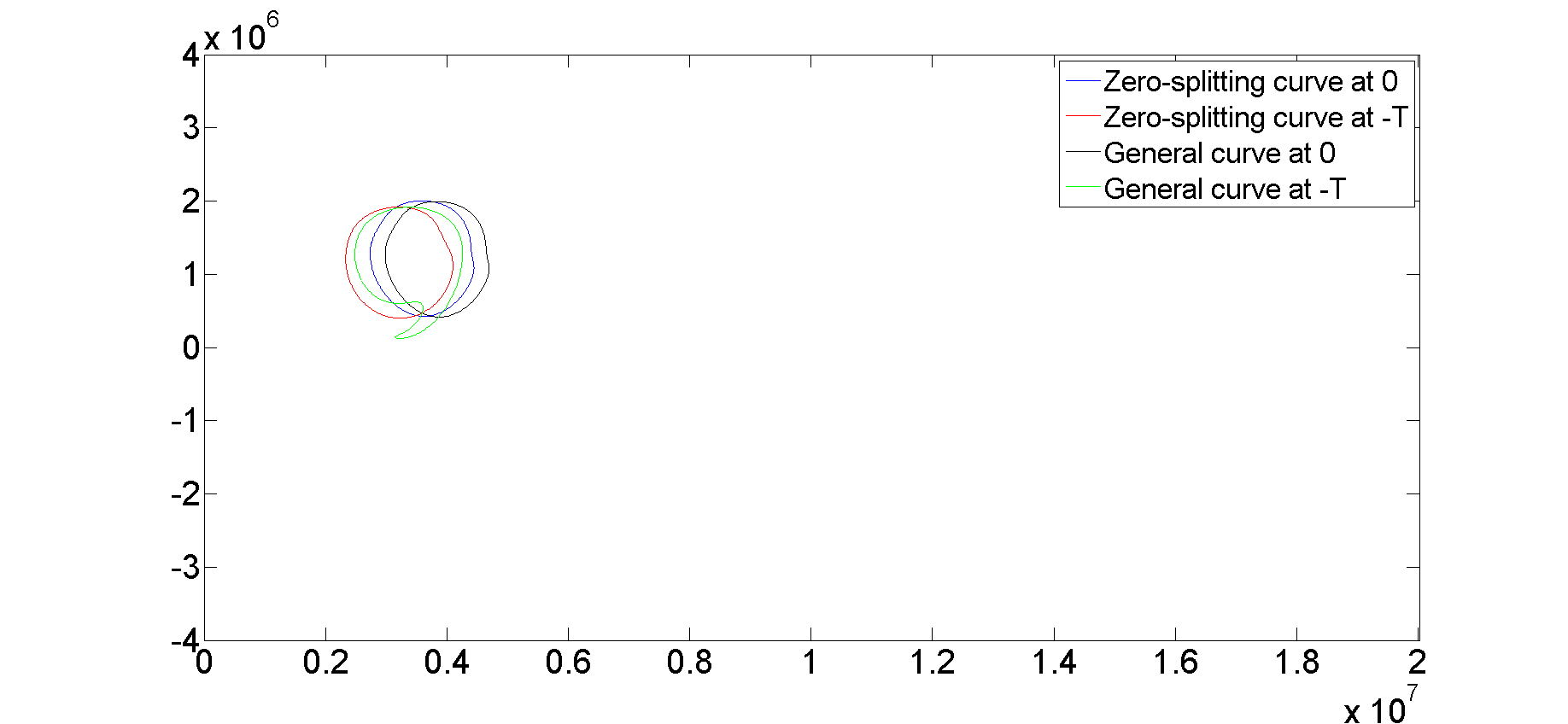}}  \\
   \subfloat[]{\label{fig:tiger}\includegraphics[width=1 \textwidth]{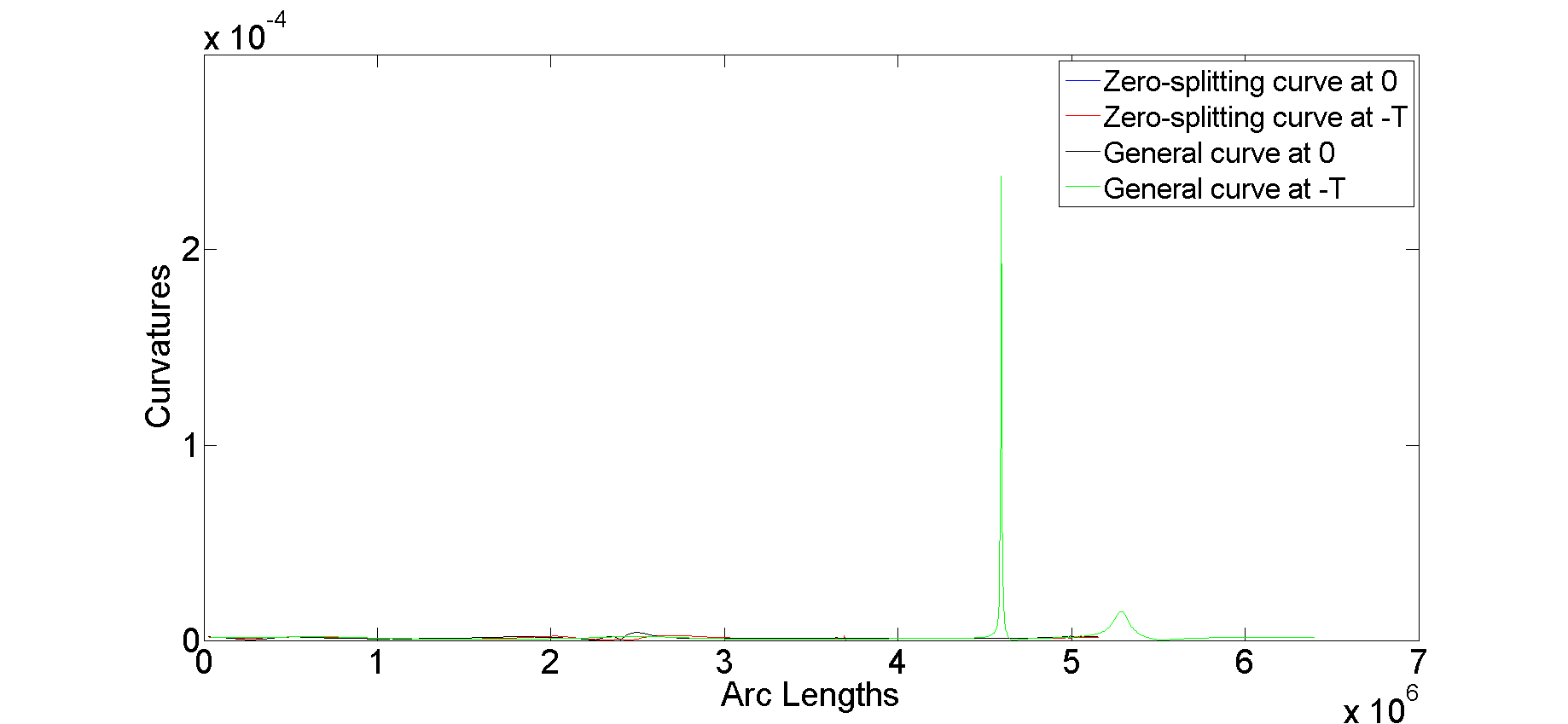}}  
  \caption{(a) The blue curve is a zero-splitting curve $C$ at time $0$ and the black curve is a general curve $\tilde{C}$ which is a slightly shifted version of the blue curve. We can see that at time $-T$, comparing to the general curve, the zero-splitting curve roughly keeps its curvatures, and arc length better. (b) In a curvature as a function of arc length, $k(s)$, notice that the general curve (black curve above) curvature changes dramatically, as seen by the green bump. Likewise, the green curve extends on x-axis far from the other three, meaning that the arc length of the general curve grows much more than the zero-splitting curve.  }
  \label{FoliationAndCurves4}
\end{figure*}

\section{On Continuation Curves of Zero-Splitting}\label{continuation1}

We enlist the implicit function theorem for proof of existence of  curves of zero-splitting and construction.  
We cite the planar implicit function theorem as follows, as the planar version is sufficient for the purposes of this paper.


\begin{theorem}[Implicit Function Theorem]\cite{M}
If $F:E\rightarrow {\mathbb R}$ is continuously differentiable in a domain $E\subset {\mathbb R}$, and an equation $F(x,y)=c$ has some point $(x_0,y_0)\in E$ such that $F(x_0,y_0)=c$, and $\frac{\partial F}{\partial y}(x_0,y_0)\neq 0$, then there exists a neighborhood $U\subset E$ of $x_0$ and a function $y=y(x)$ in this neighborhood such that $y(x_0)=y_0$ and $F(x,y(x))=c$ for all $x\in U$.
\label{IFT}
\end{theorem}

If the vector field has enough regularity to ensure that the function $\theta(z, t)$ is at least $C^1$, then by the implicit function theorem, we have the following,

\begin{theorem}[Continuation Theorem] \label{continuation} 
 The set of $z=(x,y)$ with $\theta(z, t)=0$ are a set of $C^1$ curves, which can be written as $C^1$ functions such as $y=g_1(x)$ or $x=g_2(y)$ of a finite $t$, 
which depends on $\theta_y \ne 0$ or $\theta_x \ne 0$.  Furthermore, $dy/dx=-\theta_x/\theta_y$, for a given initial condition $z_0$ has a solution $g_1$ or likewise $dx/dy=-\theta_y/\theta_x$ respectively
\label{Continuation}
\end{theorem}

By theorems in Sec.~\ref{seccurvesofstableunstable}, these curves relate to shape coherence.
The following Algorithm offers a numerical method to obtain these boundaries.




\subsection{Numerical Continuation}\label{numericalcont}
Theorem \ref{IFT}-\ref{Continuation} leads readily to a numerical continuation method to find  zero-splitting curves, by adaptation of the idea of continuation by using the differential equation from the implicit function theorem.
The implicit function theorem gives that solving the initial value problem, 
\begin{equation}\label{odef}
\frac{dy}{dx}=-\frac{\theta_x}{\theta_y}, (\mbox{or possibly } \frac{dx}{dy}=-\frac{\theta_y}{\theta_x} \mbox{ if tracking with respect to the other coordinate}), \space \space \theta(z_0,t)=0, 
\end{equation}
from a seed point $z_0(x_0,y_0)$ where $ \theta(z_0,t)=0$.
A solution of this initial value problem represents a subset segment of the curve of zero-splitting of the stable and unstable foliations.  Numerical continuation is a common theme in applied mathematics, and specifically it is used in dynamical systems in numerical bifurcation theory, \cite{continuationmethod}.  There are now many sophisticated algorithms for this purpose, and we will describe only the simplest rudimentary variation of such an algorithm here. 

In practice seeding the initial value problem requires finding at least one good point $(\tilde{z_0})$ which is near a zero point, which by continuity of $\theta(z,t)$, if $\|({z_0})-(\tilde{z_0})\|<\delta$, then $|\theta(\tilde{z_0},t)|<\epsilon$.  Such a point $(\tilde{z_0})$ for an initial rough threshold $\epsilon_1$ can be found by essentially a random search in the domain of interest, and then an improved seed representing a smaller threshold $0<\epsilon_2<\epsilon_1$, can be found by numerical optimization (root finding-like) algorithm applied to $ \theta(z,t)=0$ with that rough seed $\tilde{z_0}$, to the useful tiny threshold $\epsilon_2$.  For example, we choose $\epsilon_1=1 \times 10^{-2}$, and $\epsilon_2=1 \times 10^{-10}$.  Then a numerical ODE solver continues along solutions of Eq.~(\ref{odef}), but with the caveat that at each step of the numerical ODE solver, a corrective step must be taken by repetition of the root finding algorithm to the threshold $\epsilon_2$, as a search purely along the $y$ variable when solving along $\frac{dy}{dx}=-\frac{\theta_x}{\theta_y}$.  This together with the caveat that the roles of $x$ and $y$ can reverse when reaching singularities $\theta_x=0$ or $\theta_y=0$ representing points where the curve can become multiply valued as a function over $x$ (or $y$). So stating below as a prediction then correction type algorithm, which we describe as steps rather than a complete algorithm.

\begin{enumerate}

\item Find a rough seed in the domain $\tilde{z}_0\in \Omega$, for a given rough threshold, $\epsilon_1>0$, and let $n=1$.

\item An improved seed is found $z_n=(x_n,y_n)$, to the required precision $0<\epsilon_2<\epsilon_1$, such that $|\theta(z_n,t)|<\epsilon_2$ by the Trust-Region Dogleg Method, \cite{Sorensen}. 

\item Make a predictive step by Euler's method on Eq.~(\ref{odef}), $\tilde{y}_{n+1}=y_{n}+h -\frac{\theta_x}{\theta_y}(z_n,t)$, and $x_{n+1}=x_n+h$, for a chosen small $h>0$. (or $\tilde{x}_{n+1}=x_{n}+h -\frac{\theta_y}{\theta_x}(z_n,t)$ and  $y_{n+1}=y_n+h$ if the roles are reverse.) 

\item Make a corrective step by the root finder, $\tilde{z_0}\rightarrow z_0$ again to precision $\epsilon_2$ by the Trust-Region Dogleg Method but
this time while holding the dependent variable of the ODE constant.  (If solving $\frac{dy}{dx}=-\frac{\theta_x}{\theta_y}$ then $x$ is the active variable, and vice-versa if solving       $\frac{dx}{dy}=-\frac{\theta_y}{\theta_x}$ ).

\item If the right hand side of the differential equation develops a singularity, $\theta_x(z_n)\approx 0$, or approximately so, ($|\theta_x(z_n)|<\epsilon_3$ for some small $\epsilon_3>0$), (or $\theta_y(z_n)\approx 0$ depending on which of the two ODEs is currently being used in Eq.~(\ref{odef}),  then reverse the roles of $x$ and $y$, meaning begin tracking the other version of the continuation equation in Eq.~(\ref{odef}).

\item Repeat step 3 until a stopping criterion is reached.

\item Stopping criteria, keep repeating Step $4$ until it cannot find zero points in the trust region or find a zero points already exists.

\item Connect the zero splitting points from Step $4$, after Step $5$, we get a curve for one seed. 

\item Connecting gap criteria, if the gap is smaller than a given distance $l_{max}$, 
we just connect it and claim the internal region of these curves is a coherent structure candidate.

\end{enumerate}

Stopping criterion generally is the result of  the inability of the root finder to find a root to the required precision.  Variations of this could proceed by adaptively reducing step size $h$ and retrying the test step as possibly the loss of the curve is the reason of the inability to find a root and in such case the method could have stepped past a point of singularity requiring a role reversal. Note that such a phenomena is sensitive to step size and generally we choose small step sizes.  Because of the role of the corrector step, we have not been motivated to choose a higher order numerical ODE integrator but again we emphasize that there are many more sophisticated continuation algorithms available, \cite{continuationmethod}.
We have found that even this simple algorithm gives excellent smooth curves to high precision, and quickly, s evidenced by the examples, Figs.~\ref{RW2}-\ref{RW2nd2} and \ref{DG1}-\ref{DG2}.  Finally note that since there are generally many zero splitting curves, repeated initial seeding for Step 1 above can proceed by randomly choosing many initial conditions $\tilde{z}_0\in \Omega$, in an attempt to satisfy the chosen rough threshold, $\epsilon_1>0$.


\section{Examples}\label{examples}
In this section, we apply our methods to the Rossby wave system and the double gyre system, both of which have become benchmark examples for studying almost invariance, coherence and transport, 
\cite{FSM, FK,MB, BN}.

%

\subsection{ An Idealized Stratospheric Flow}
Consider the Hamiltonian system
\begin{eqnarray}
dx/dt=-\partial \Phi / \partial y \nonumber \\
dy/dt=\partial \Phi / \partial x,
\end{eqnarray}
where
\begin{eqnarray}
\Phi(x,y,t)=c_3y&-&U_0Ltanh(y/L)   \\
&+&A_3U_0Lsech^2(y/L)cos(k_1x) \nonumber \\
 &+&A_2U_0Lsech^2(y/L)cos(k_2x-\sigma _2t) \nonumber \\
&+&A_1U_0Lsech^2(y/L)cos(k_1x-\sigma _1t)  \nonumber 
\label{ROWA}
\end{eqnarray}
This quasiperiodic system represents an idealized zonal stratospheric flow \cite{RBB,FSM}. 
There are two known Rossby wave regimes in the system. 
We will show two cases with different parameters.
\begin{enumerate}
\item 
Let $U_0=63.66, c_2=0.205U_0, c_3=0.7U_0, A_3=0.2, A_2=0.4, A_1=0.075$ and the other parameters in Eq. 
\ref{ROWA}  be the same as stated in \cite{RBB}. 
We set the time epoch $T=3 \ days$. 
By the parameters, we emphasize the main partition and small gaps between zero-splitting curves in the Rossby wave.

At first, we generate a uniform $2000 \times 200 $ grid of the domain $[0, 6.371\pi \times10^6] \times [-2.5\times10^6, 2.5\times 10^6]$. 
Fig. \ref{RW1} (a) shows the finite-time stable and unstable foliations for each of these $4\times 10^5$ points in the domain. 
Then we obtain amongst these, 3292 points with foliations' angles smaller than $\epsilon_1=10^{-2}$. 

Fig. \ref{RW2} (a) and (c) are the zero-splitting curve that result from the continuation algorithm in Sec.~\ref{numericalcont}, but then shown as evolved at different times, $T=-3 \ days$ and $T=3 \ days$. 
Notice the small gaps in the middle region can be connected, so we can get the middle partition of the zonal flow.
However, the small gaps indicate more details of the mixing behaviors of the flow.  For comparison we have included in Figs.\ref{RW2} (b) and \ref{RW2} (d) the results using the Frobenius-Perron operator based coherent pairs method \cite{FSM}.

\begin{figure*}
  \centering
            
   \subfloat[t=-3 \ days]{\label{fig:tiger}\includegraphics[width=.55 \textwidth]{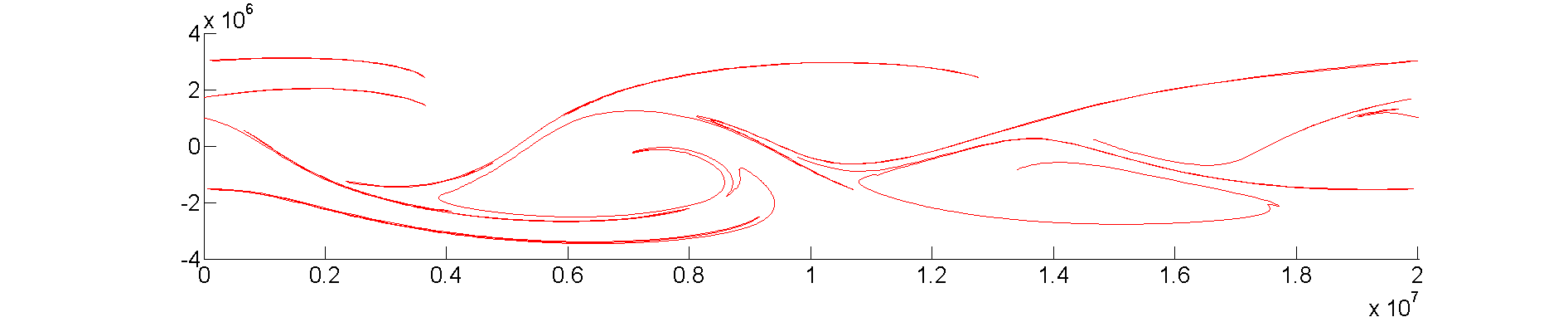}}   
   \subfloat[t=3 \ days]{\label{fig:tiger}\includegraphics[width=.55 \textwidth]{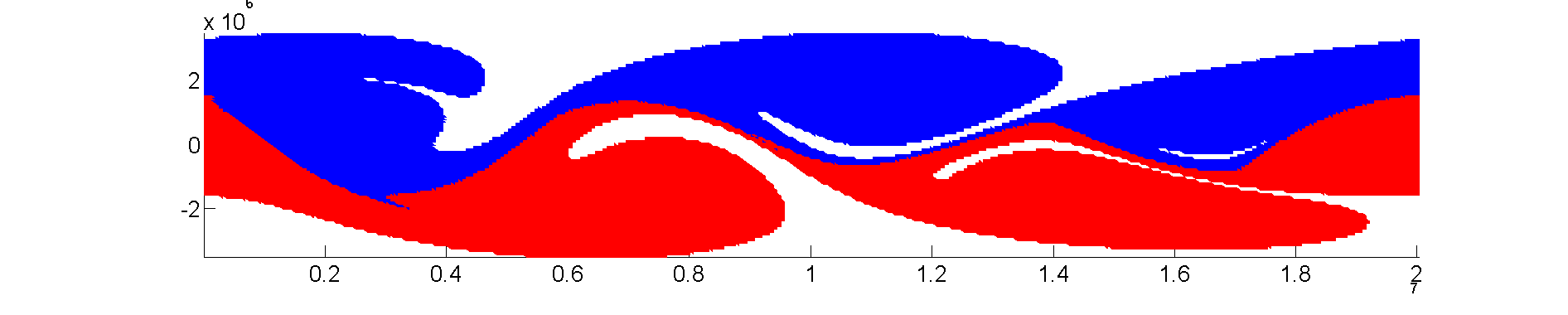}}  \\
   \subfloat[t=-3 \ days]{\label{fig:mouse}\includegraphics[width=.55\textwidth]{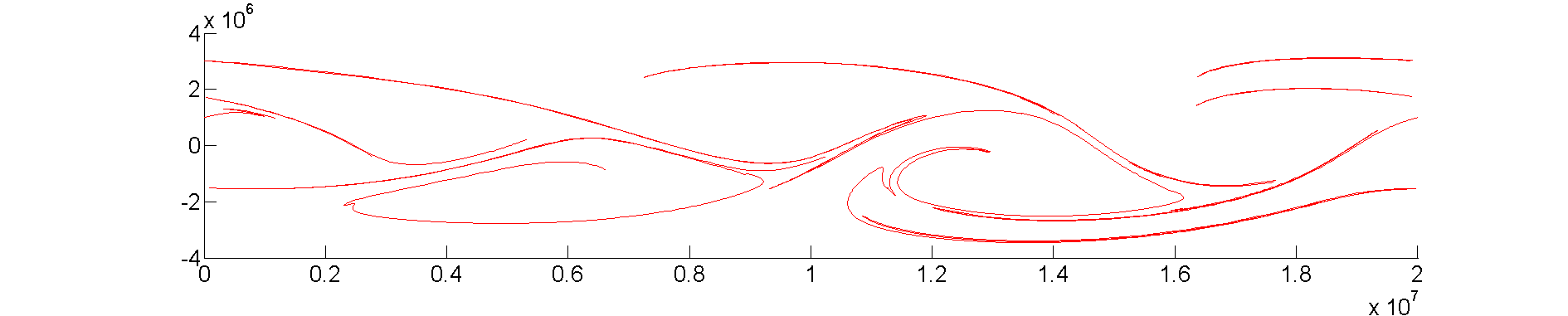}}  
   \subfloat[t=3 \ days]{\label{fig:mouse}\includegraphics[width=.55\textwidth]{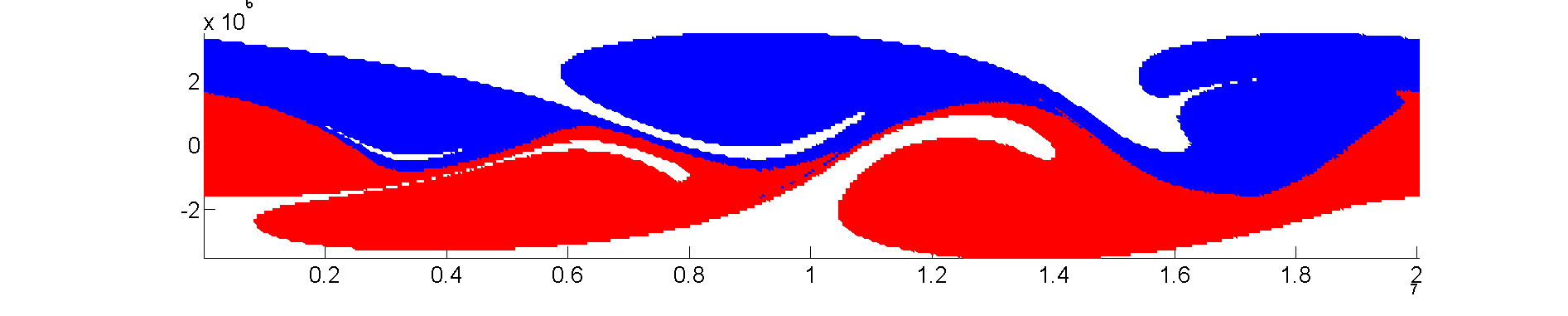}}  \\

  \caption{(a) and (c) are curves with zero splitting angles at different times, the small gaps between curves indicate the leaking points of the middle partition from coherent pair method, which is (b) and (d).}
  \label{RW2}
\end{figure*}

\item 
Let $U_0=44.31, c_2=0.205U_0, c_3=0.461U_0, A_3=0.3, A_2=0.4, A_1=0.075$ and the other parameters are the same as the above.
For this example, we change $T$ from $3 \  days$ to $5 \ days$.
We describe the whole partition and elliptic islands by the paramters.

We choose the same grid to seed initial conditions as above to apply our method. 
Figs. \ref{RW2nd1} (a) and (c) are the initial status and final status of the resulting zero-splitting curves. 
Notice the strong similarity here to the results shown in Figs. \ref{RW2nd1} (b) and (d) of the relatively coherent pairs method results from \cite{MB} that specialized \cite{FSM} to a hierarchical partition.
It is apparent that in addition to the larger scaled north-south barriers, that the interior elliptic island-like structures are also found by both methods.
Fig. \ref{RW2nd2} shows a movie of the zero-splitting curves, that illustrate directly that as time proceeds it is apparent that the shapes hold together as was the original motivation.  This is a visual presentation of the shape coherence.

\item

 It is interesting to inspect the  foliations geometry in more detail.
 Fig. \ref{RW1} (b) describes how angle changes with $x-$coordinates for a fixed $y-$coordinate.
We focus on a small region of the domain, Fig.~\ref{RW1} (c), and the angle function in this restricted domain Fig.~\ref{RW1} (d).
We believe that the fast switching behaviors between $0$ and $90$ degrees indicate a efficient mixing system.  Note that as the time epoch $T$ is increased, the angle function develops increasing variation, as the foliation switches direction more and more quickly. The eventual development in the limit of long time windows would be that the foliations would change direction infinitely many times in a finite sized small domain representing the behavior of the stable and unstable manifolds that are known to accumulate in horseshoe-like trellis structures in many common chaotic dynamical systems.  On the other hand, for small time epochs, we can study how zero-splitting of the foliation first develops. 
Fig. \ref{RWAngles} shows how angle changes by time $T$ for a fixed $y-$coordinate.  From a practical standpoint, we note how  those zeros of $\theta(z,t)$ that develop first seem to generally correspond to primary partitioning of the phase space, while those zeros that develop only at larger times seem to correspond to smaller-scaled shape coherent structures.

\end{enumerate}

\begin{figure*}
  \centering
            
   \subfloat[t=-5 \ days]{\label{fig:tiger}\includegraphics[width=.55 \textwidth]{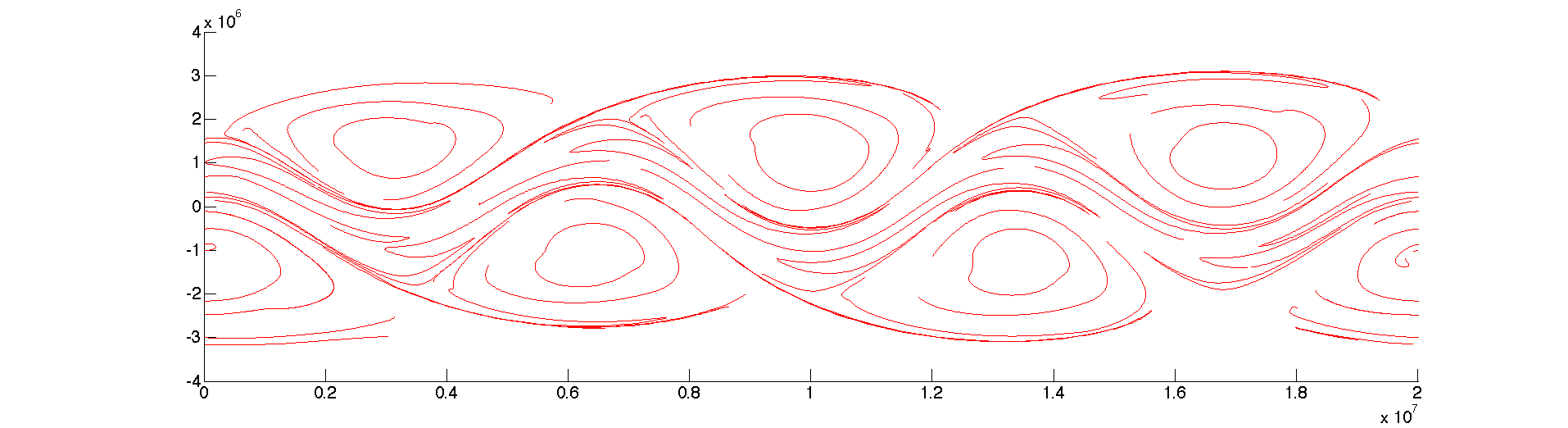}}   
   \subfloat[t=5 \ days]{\label{fig:tiger}\includegraphics[width=.55 \textwidth]{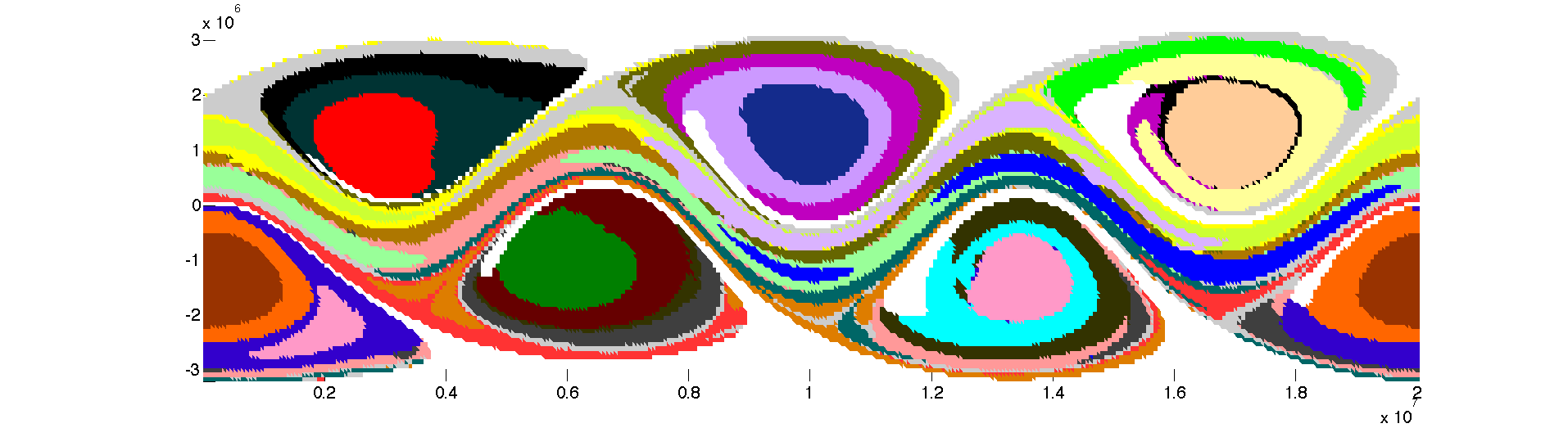}}  \\
   \subfloat[t=-5 \ days]{\label{fig:mouse}\includegraphics[width=.55\textwidth]{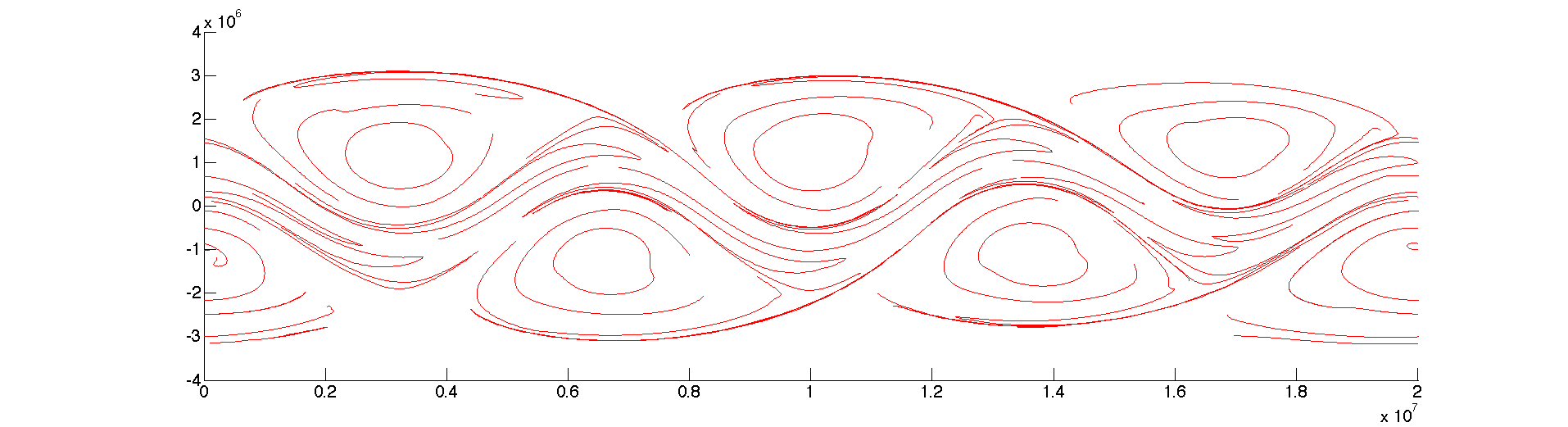}}  
   \subfloat[t=5 \ days]{\label{fig:mouse}\includegraphics[width=.55\textwidth]{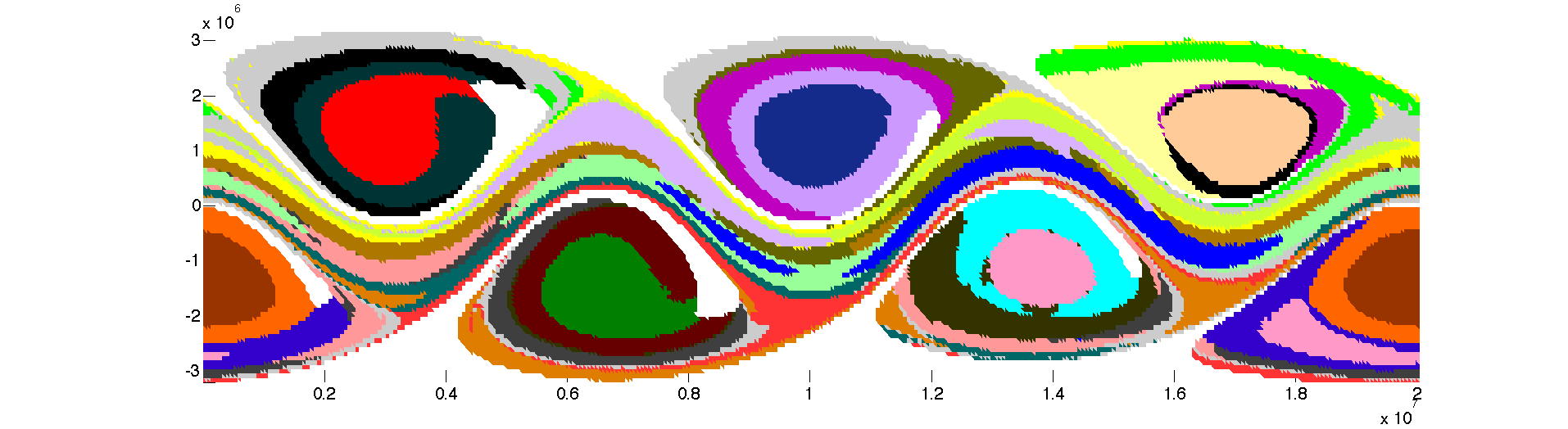}}  \\

  \caption{(a) and (c) are curves with zero splitting angles at different times, the elliptic islands are consistent with the results from coherent pair method, which is (b) and (d).}
  \label{RW2nd1}
\end{figure*}

\begin{figure}
  \centering           
  \subfloat[A snapshot of the movie]{\label{fig:tiger}\includegraphics[width=.55 \textwidth]{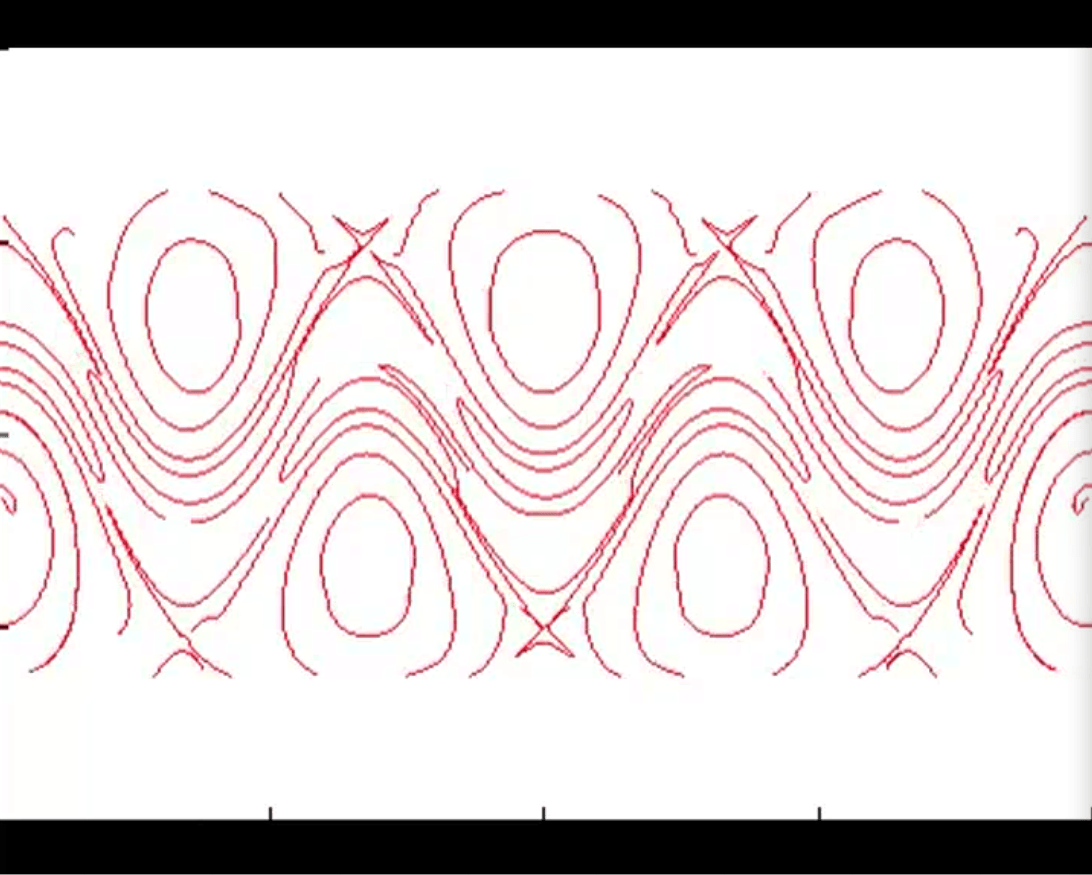}}  
  \caption{The movie of the stable and unstable foliations with non-hyperbolic splitting for zonal flow from $T=-5 \ days$ to $T= 5 \  days$.}
  \label{RW2nd2}
\end{figure}

\begin{figure*}
  \centering
   \subfloat[Stable and unstable foliations for whole domain]{\label{fig:tiger}\includegraphics[width=.9 \textwidth]{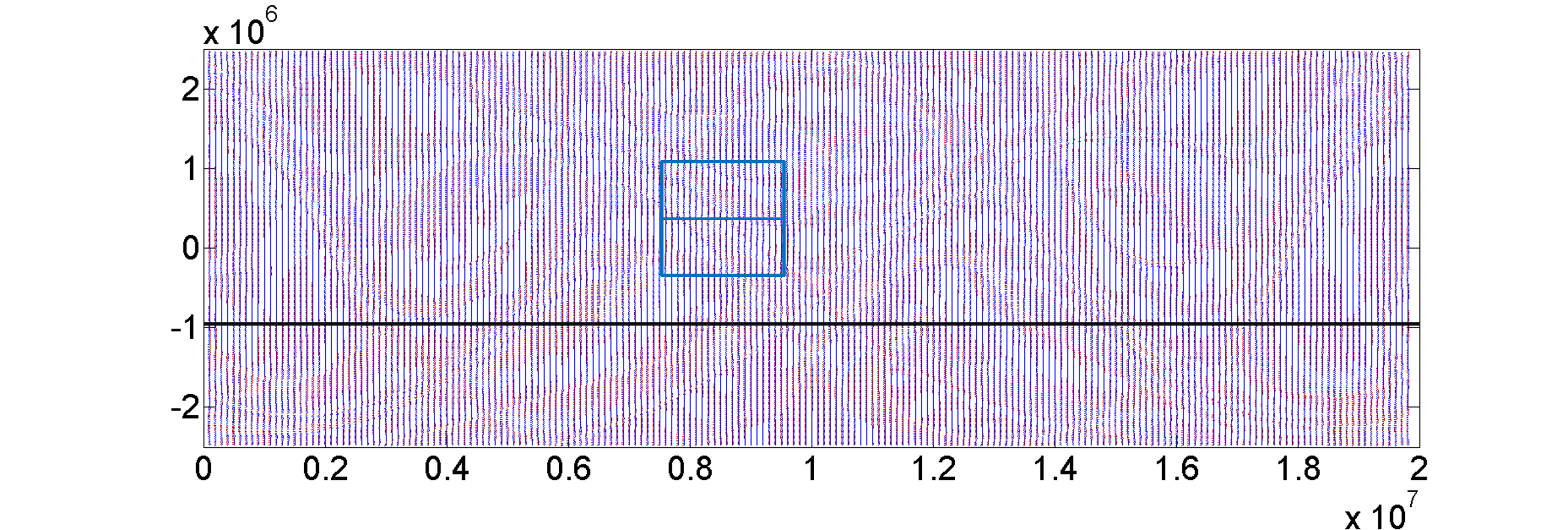}}   \\
   \subfloat[Angles vs $x$ for a fixed $y$]{\label{fig:tiger}\includegraphics[width=.9 \textwidth]{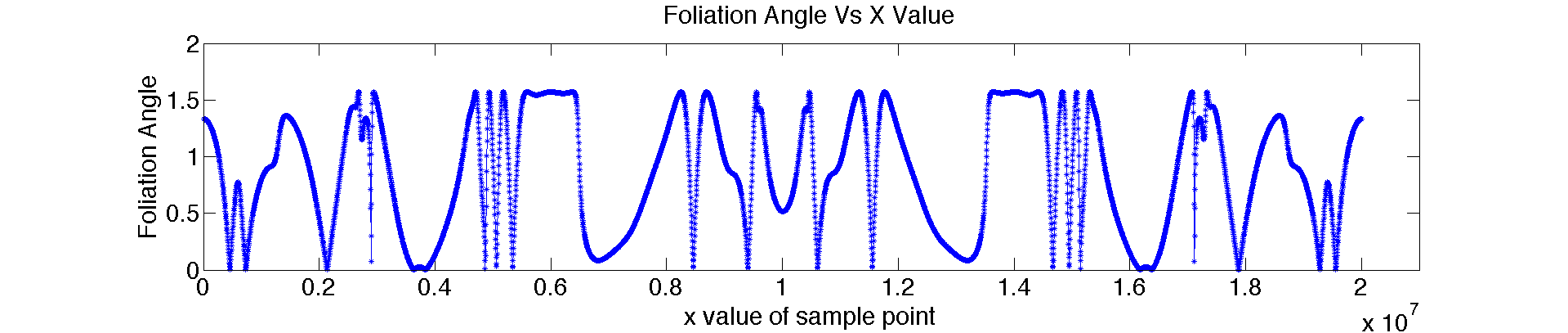}}  \\
   \subfloat[Foliations in the blue rectangular]{\label{fig:mouse}\includegraphics[width=.6\textwidth]{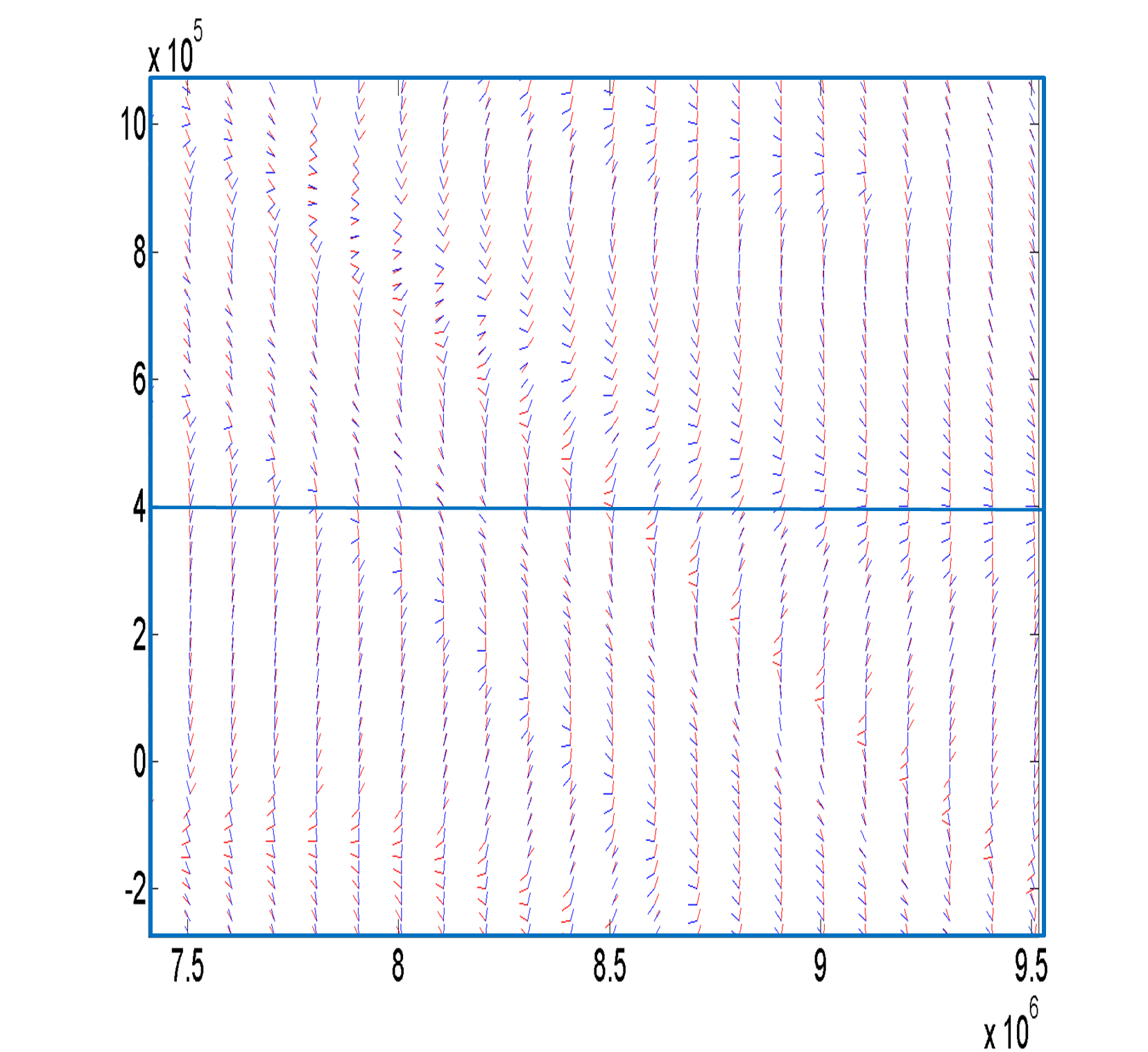}}  \\
   \subfloat[Angles vs $x$ in the small blue region]{\label{fig:mouse}\includegraphics[width=.6\textwidth]{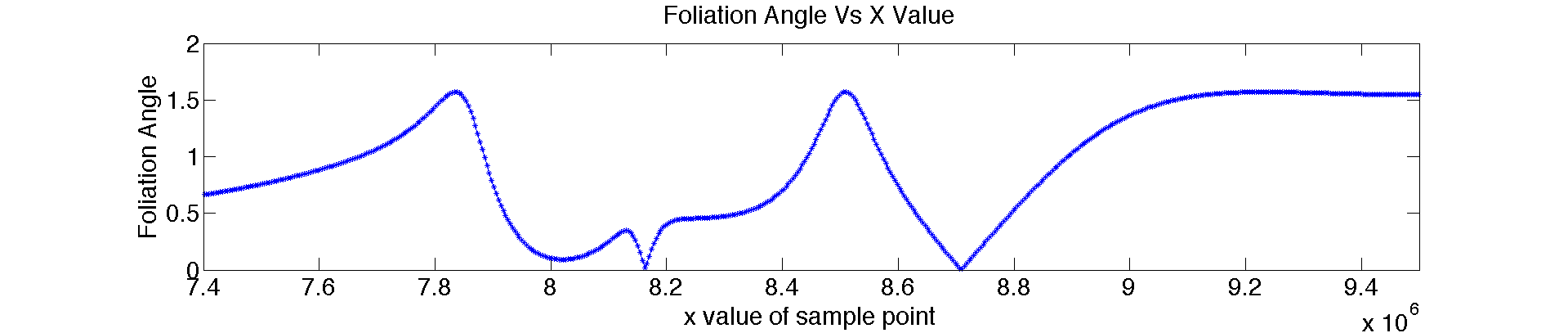}}

  \caption{(a) is the stable and unstable foliations of Rossby wave at time $t=0$, the wavelike stuctures result from the differences among angles; 
(b) is the plot of foliations angles vs $x$ value for a fixed $y$; 
(c) is how foliations look like in  the small blue area of the whole domain; 
(d) is the angle changes with $x$ for a fixed $y$ for the small blue region.}
  \label{RW1}
\end{figure*}

\begin{figure*}
  \centering            
   \subfloat{\label{fig:tiger}\includegraphics[width=.85 \textwidth]{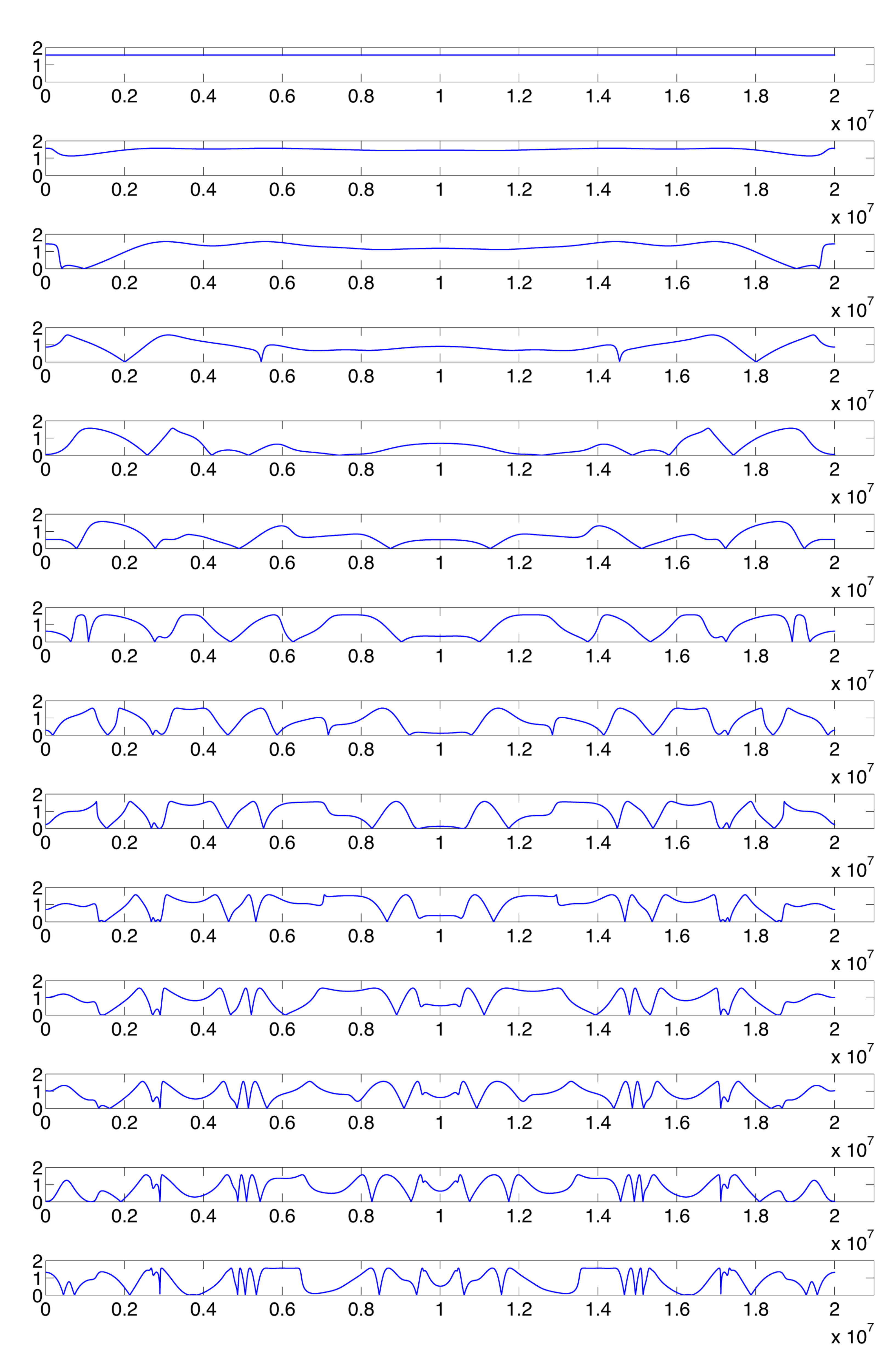}}  
  \caption{For $y=-1 \times 10^{6}$, it shows the angles between foliations for different uniformly spaced times from $T=0$ to $T=\pm 3 \ days$ versus the $x-$ coordinates of Rossby Wave. At the beginning, all the angles are equal 90 degree. Then the angles on the two sides become smaller, and more small angles keep emerging for longer time. At last, we get the figure which is in Fig. \ref{RW1} (b). }
  \label{RWAngles}
\end{figure*}

%

\subsection{The Nonautonomous Double Gyre.}
Consider the nonautonomous double gyre system,
\begin{eqnarray}
&& \dot{x}=- \pi A \sin ( \pi f(x,t)) \cos(\pi y) \nonumber \\
&&\dot{y}=\pi A \cos ( \pi f(x,t)) \sin(\pi y) \frac{df}{dx}
\end{eqnarray}
where $f(x,t)= \epsilon \sin(\omega t) x^2 + (1-2 \epsilon \sin (\omega t)) x$, $\epsilon = 0.1$, $\omega=2 \pi/10$ and $A=0.1$.
See \cite{SLM,FK}. 
Let the initial time be $t_0=0$ and the time epoch to build the folations is $T=10.$
For double gyre, we focus on the main partition and the closed curves. 

We calculate the finite-time stable and unstable  foliations 
on a uniform $2000 \times 200$ grid of the domain $[0, 2] \times [0, 1]$. 
See Fig. \ref{DG1} (a) shows the zero-splitting curves at time $T=-10$ and  Fig. \ref{DG1} (b) shows the same curves evolved $2T$ to time $T=10$.
Investigating finer scaled structures, 
Fig. \ref{DG1} (c) and (d) describe two small different regions in the double gyre system. 
In particular, those zero-splitting points $z$ with angle functions $\theta(z,T)$ with large variation (with respect to $z$), which change quickly between $0$ and $\pi /2$ in (d) show 
the well-known fast mixing middle region of double gyre.  These points are similar to those seen in the Rossby system in Fig.~\ref{RWAngles}.

\begin{figure*}
  \centering

  \subfloat[Foliations of whole domain of Double Gyre]{\label{fig:gull}\includegraphics[width=1\textwidth]{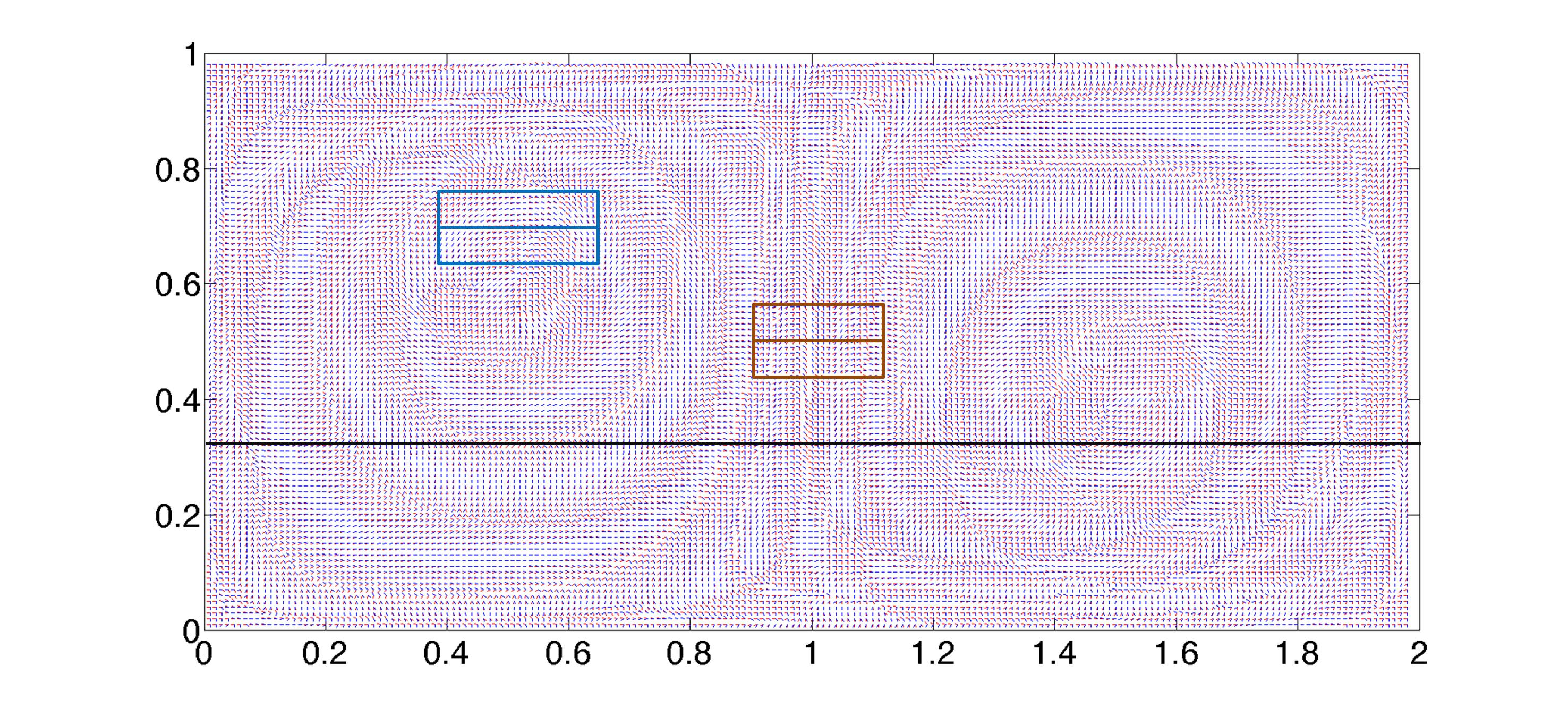}}    \\    
  \subfloat[Foliations' Angle on the Black Line]{\label{fig:gull}\includegraphics[width=1\textwidth]{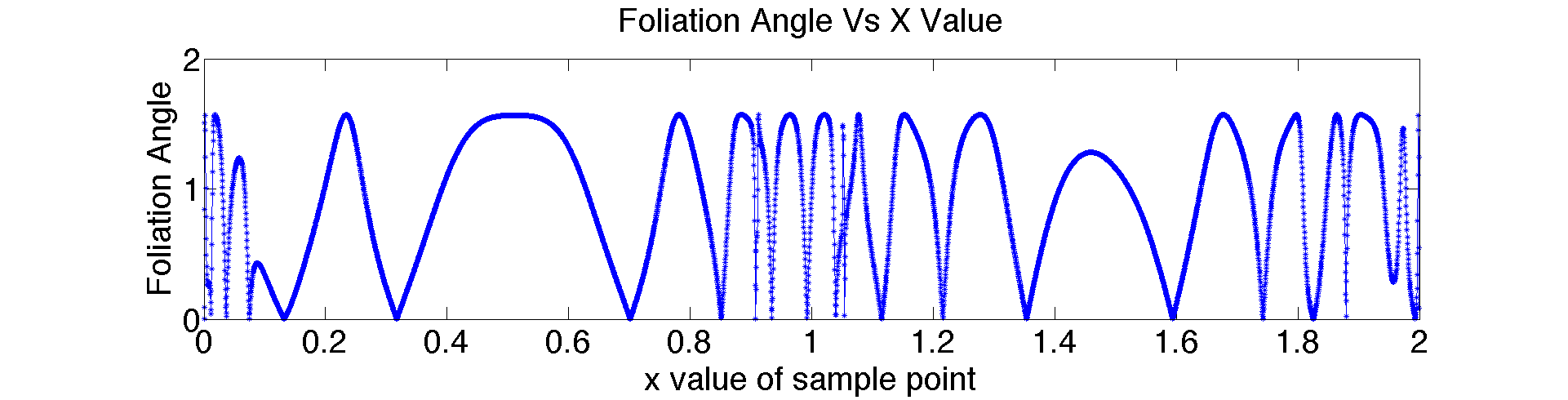}}      \\       
  \subfloat[Foliations in the blue rectangular]{\label{fig:tiger}\includegraphics[width=.5\textwidth]{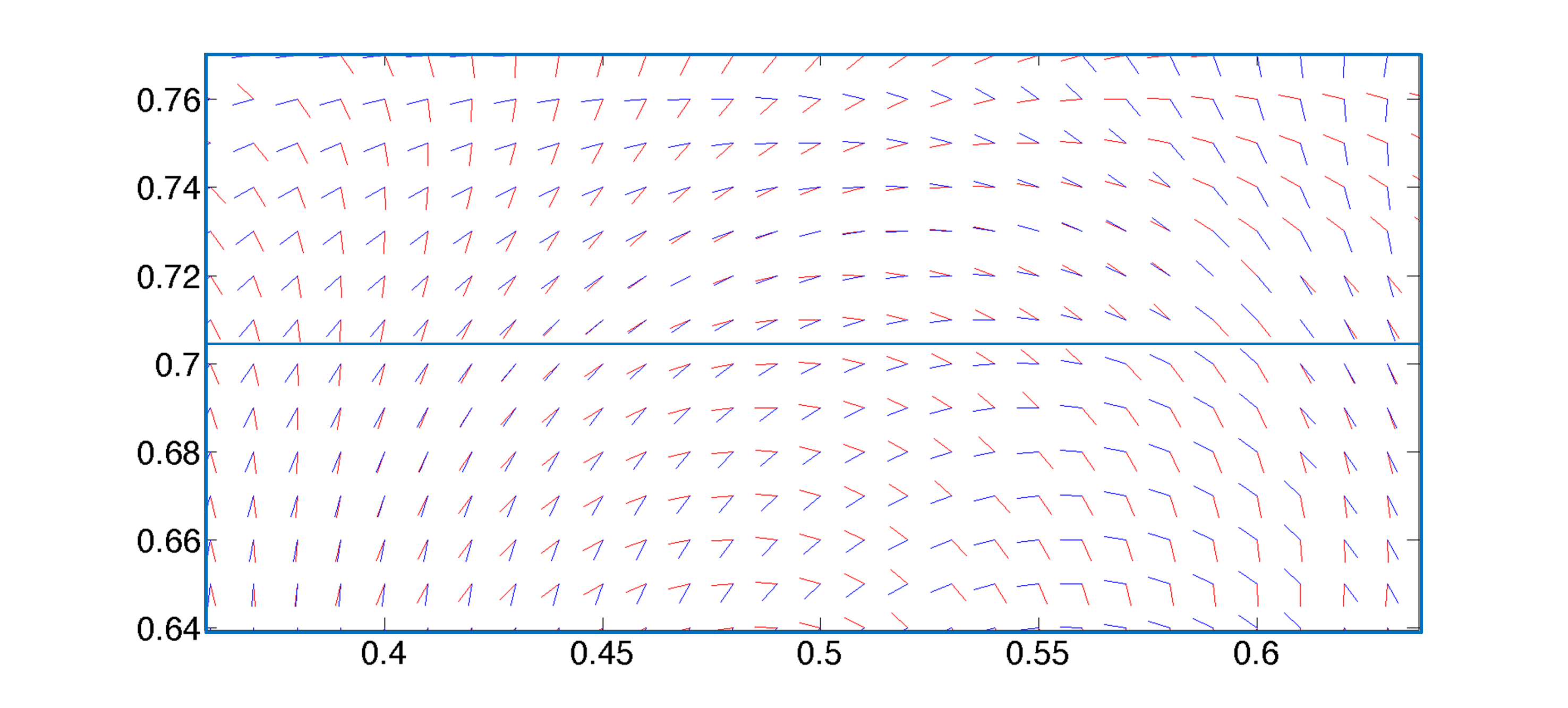}} 
  \subfloat[Foliations in the brown rectangular]{\label{fig:tiger}\includegraphics[width=.5\textwidth]{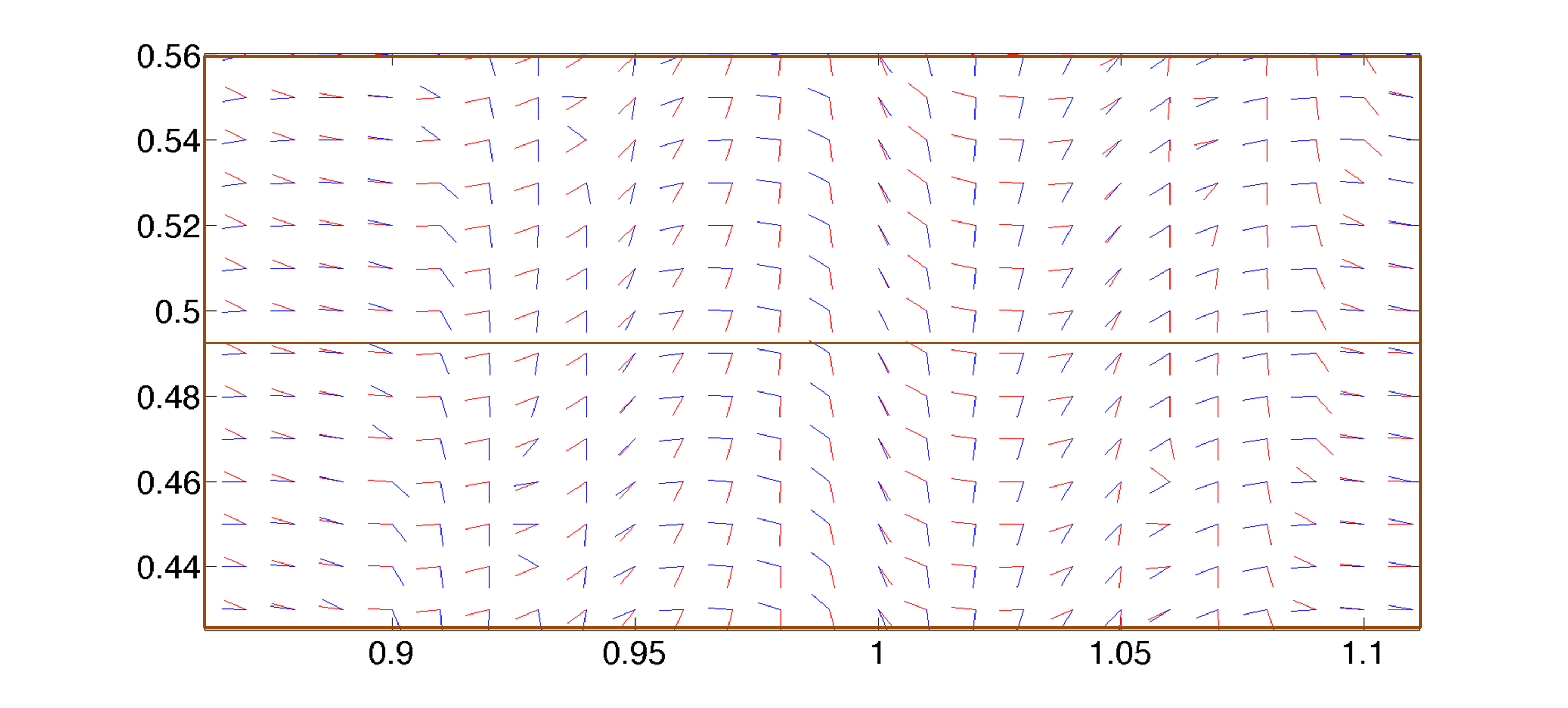}} \\
  \subfloat[Angles in the blue line]{\label{fig:tiger}\includegraphics[width=.5\textwidth]{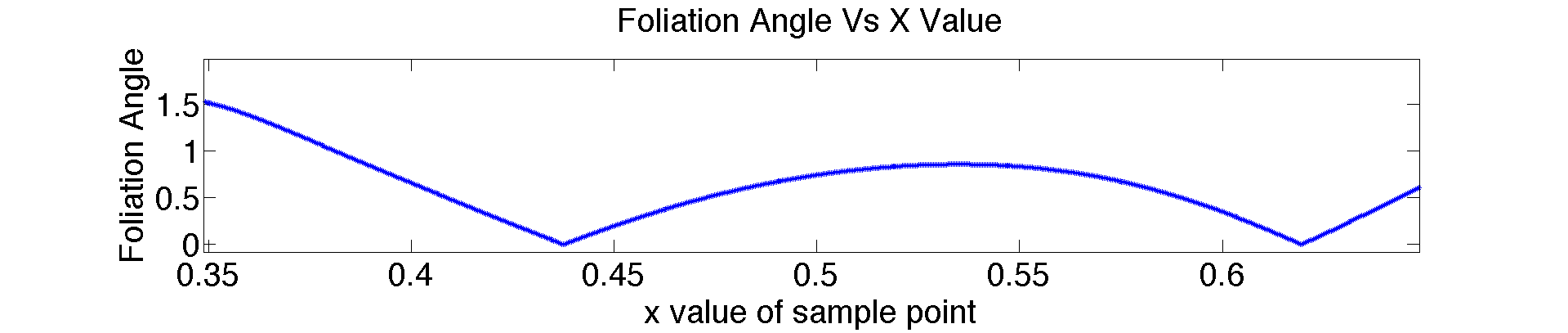}} 
  \subfloat[Angles in the brown line]{\label{fig:tiger}\includegraphics[width=.5\textwidth]{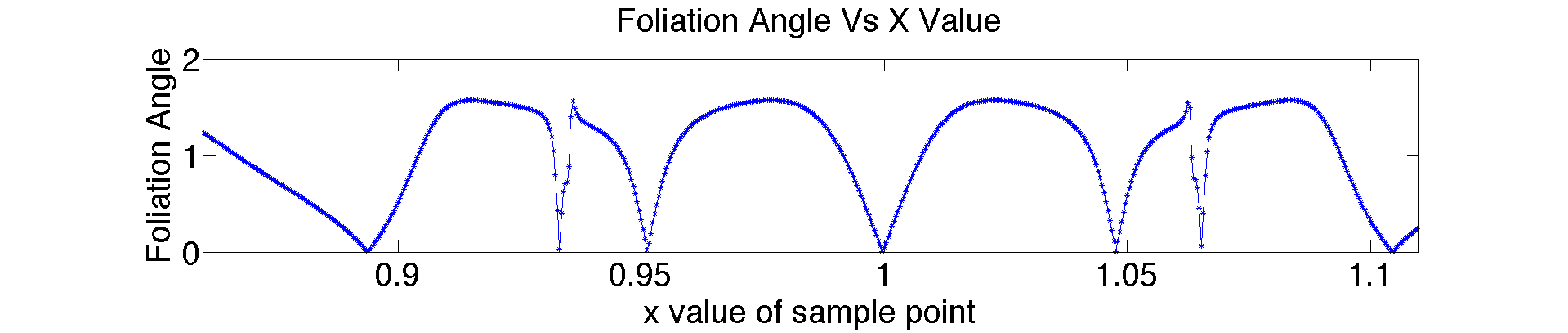}}  

  \caption{}
  \label{DG1}
\end{figure*}

\begin{figure*}
  \centering
    
  \subfloat[Main partition at T=-10]{\label{fig:gull}\includegraphics[width=.5\textwidth]{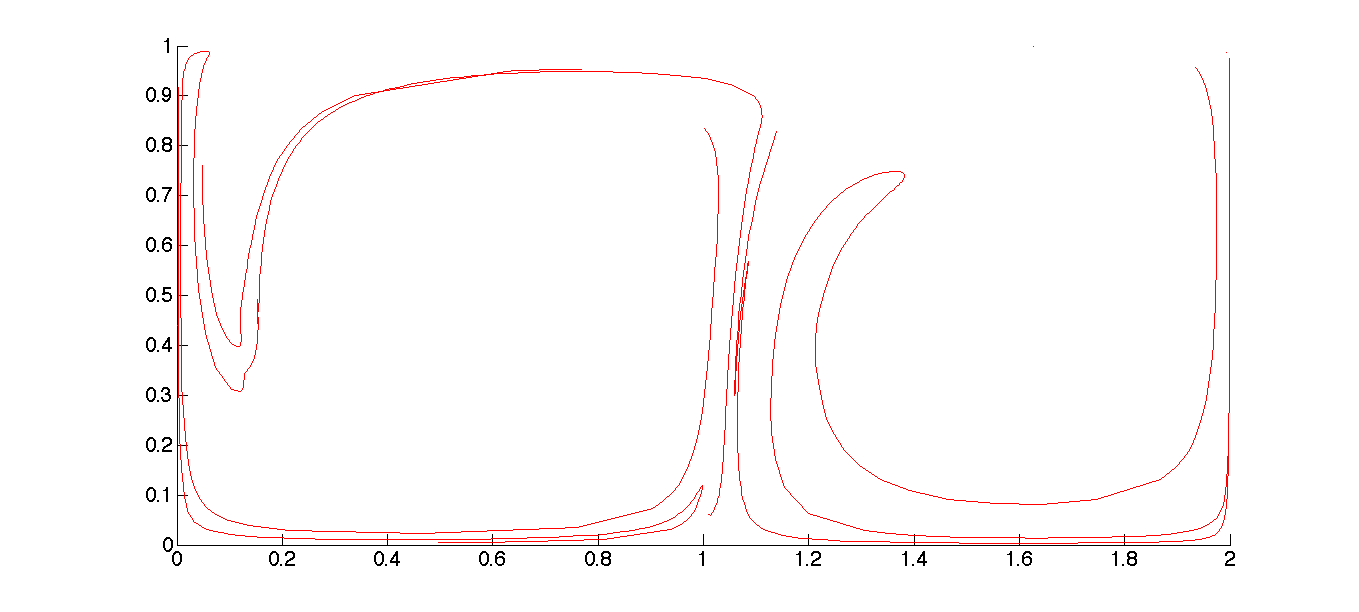}}       
  \subfloat[Main partition at T=10]{\label{fig:tiger}\includegraphics[width=.5\textwidth]{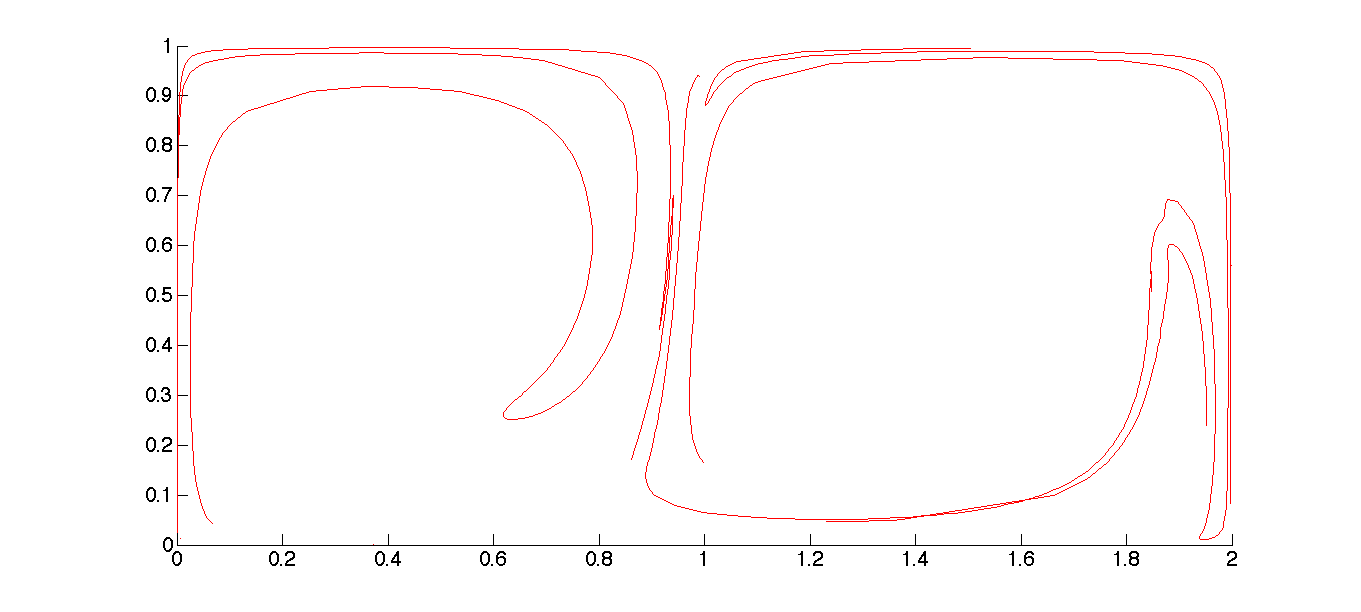}}   \\
  \subfloat[Islands partition at T=-10]{\label{fig:tiger}\includegraphics[width=.5\textwidth]{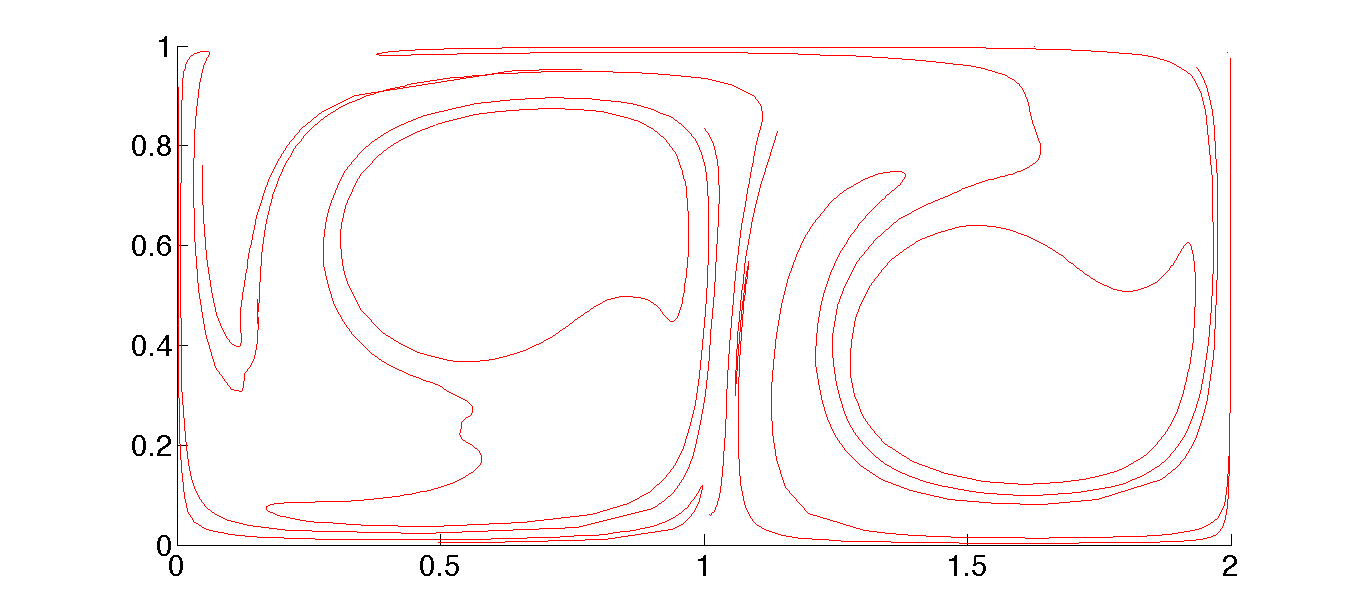} }
  \subfloat[Islands partition at T=10]{\label{fig:tiger}\includegraphics[width=.5\textwidth]{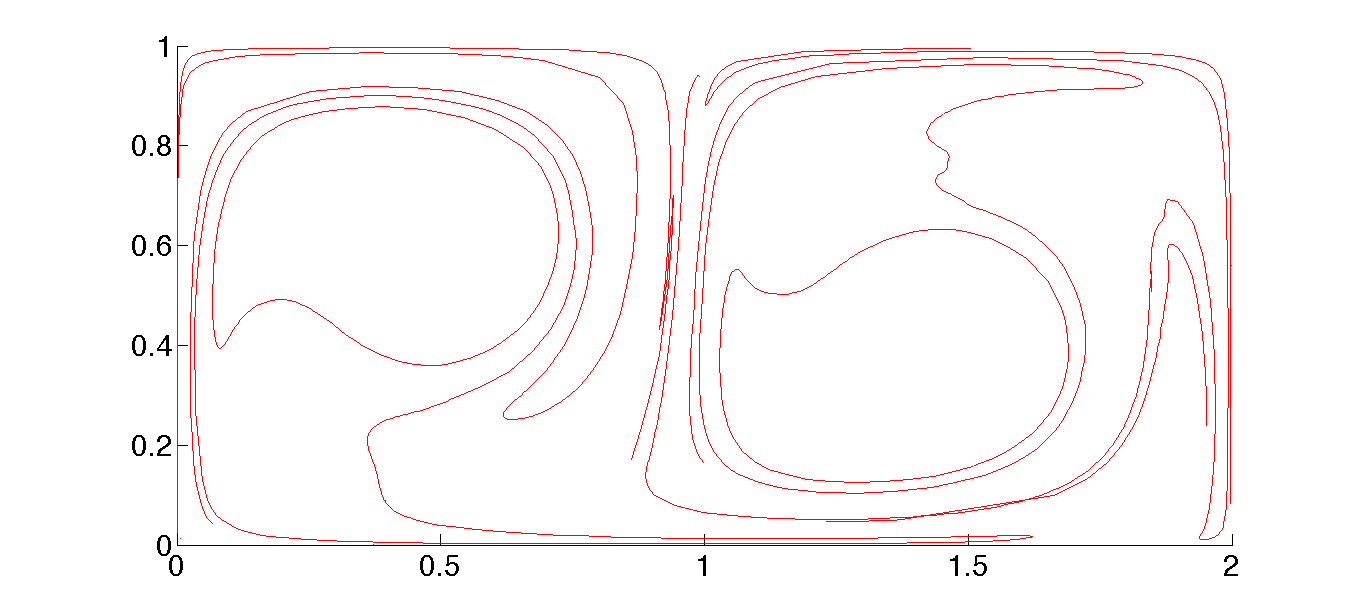}}    \\

  \caption{}
  \label{DG2}
\end{figure*}

%
\section{Conclusions}
Here we have defined a mathematically precise concept of coherence, called ``shape coherent sets," to emphasize the intuitive idea of sets that hold together in a flow, and roughly catch our eyes.  With respect to this definition, it follows that a flow that is locally approximately the same as a simpler rigid body motion is shape coherent.  Then we show that the theory of differential geometry which holds the concept of congruence between curves allows us to prove that a set whose boundary curvature changes relatively little during the finite epoch of the flow corresponds to a shape coherent set.  Thus through this equivalence, we then investigate what properties of the flow tend to cause curvature to change relatively slowly.  We show that points with tangencies between the finite time stable and unstable foliations tend to have curvature change relatively slowly.  Therefore, we search for curves of such tangency points, which are generally usually typical of shear.  We prove existence of such curves through the implicit function theorem,  which also suggests a constructive algorithm based on solving the ODE from the implicit function theorem, or many other of the typical methods for numerical continuation. Thus we develop shape coherent sets in two benchmark examples, the double gyre, and the Rossby wave system. Further, we include investigations of the intricate stable and unstable foliations topology in each of these systems, which is typical of finite versions of chaotic systems, where the complex folding of stable and unstable manifolds suggest that finite time stable and unstable manifolds will change rapidly in space.  Therefore the angles function Figs.~\ref{RW1}-\ref{RWAngles} can have highly intricate structure, and we discuss the role of primary roots for primary barriers.

In a broad sense, it might said that the phrase ``coherence" has many different meanings to many different people, some spectral, and some set oriented.  Some interpretations of the phrase allow  for sets that may significantly stretch and fold.  Clearly if a set is specifically highlighted, visually by coloring the set as a partition, and then that set is advected, then the set will appear to hold together.  However, this could be interpreted as a property of a continuous flow, that connected sets remain connected, and so a coloring alone will guide our eye for almost any coloring.  Our definition of shape coherence is not meant to capture all possible perspectives of the popular phrase ``coherence." It does however describe a phenomenon that we believe exists in a wide variety of chaotic and turbulent systems, as a type of simplicity within the otherwise complicated motion.  It is evident that it is not a property present in all possible dynamical systems, taking the Arnold Cat map for example, but it does exist and we find it to be an interesting phenomenon to uncover.  Note that shape coherence as we defined it as both spatial and time scales associated with it.  Time scales obviously describe sets that may hold together for a while before deforming, but also on a spatial scale, a rather simply set may maintain its shape whereas subsets of that set may contain highly turbulent behavior.  To that point, consider as an example for sake of discussion the Great Red Spot of Jupiter which is known to be quite complex within the great storm which maintains its approximate shape for long times.  We have shown in detail two popular examples, and we have found it to exist in many more examples including periodic and aperiodic non-autonomous flows.
Finally, we remark that fundamentally the study of coherence and specifically shape coherent sets here, is in some sense the complement to study transport, and mixing which occurs in the hyperbolic-like mixing sets.  These are fundamentally complementary questions on complement sets.


%
\section{Appendix 1: Proof of Theorems \ref{prop1}-\ref{3.6}}\label{appendixproofs}

In this section, we give  proofs of Theorems \ref{prop1} and  \ref{3.6}, which we restate here for convenience.

\begin{theorem}\label{prop1-10}(\ref{prop1})
 Given two curvature functions, $\kappa_1(s)$ and $\kappa_2(s)$ of two closed curves $\gamma_1(s)$ and $\gamma_2(s)$ with the same arc length, if $\sup_{s}|\kappa_1(s)-\kappa_2(s)| < \varepsilon$,  
then there exists, 
 \begin{equation}
\delta( \varepsilon)= s^2 \varepsilon ( \|  C_1 \|_2 + \|  C_2 \|_2 ) >0,
\end{equation}
where $C_1$ and $C_2$ are the constant vectors in Eq.~(\ref{GeneralSolution}) and $s$ is the arc length from initial point of the two lined-up curves,
such that  
\begin{equation}
\|\gamma_1(s)-\gamma_2(s) \|_2<\delta(\epsilon).
\end{equation}

\end{theorem}


\begin{proof}
Consider comparing the two curves pointwise.  While here we assume that the arc length of the two curves are the same; we will discuss the generalized case of differening arc lengths following.  
We have assumed that the curvatures are as suggested in Fig.  \ref{CurvatureProp}.
We will argue 
that distance between the two points can be controlled by their curvatures. 


Without loss of generality and for convenience of the estimates of comparison, assume that the two curves intersect at $s=0$, $\gamma_1(0)=\gamma_2(0)$, and furthermore at $s=0$ that the tangent and normal vectors are each parallel, as depicted in Fig.  \ref{CurvatureProp}.  The reason this statement can be made without loss of generality is that the Frenet-Seret formula states that a curve can be reconstructed up to initial conditions, which is the reasoning behind the definition of curve congruence that curves are called the same if they are indistinguishable up to a rigid body motion.  Therefore a rigid body motion of $\gamma_2(s)$ can be used to orient the tangent and normal at $s=0$ with those of $\gamma_1(s)$, also at $s=0$, as depicted in Fig.  \ref{CurvatureProp}. This defines constant vectors $C_1$ and $C_2$.  Note this is not a unique orientation for matching, but sufficient and convenient.  

We denote the constants as $C_1$ and $C_2$, rather than two different pairs.
So we have, 
\begin{eqnarray}\label{GeneralSolutionForTwoCurvesOneConstants}
\gamma_1(s) = \int_{0}^{s} \left ( C_1 \sin \int_{0}^{s} \kappa_1 (\sigma) d \sigma -C_2  \cos \int_{0}^{s} \kappa_1 (\sigma) d \sigma \right) d \sigma   \\
\gamma_2(s) = \int_{0}^{s} \left ( C_1 \sin \int_{0}^{s} \kappa_2 (\sigma) d \sigma -C_2 \cos \int_{0}^{s} \kappa_2 (\sigma) d \sigma \right) d \sigma
\end{eqnarray}
Then the distance between the two points $\gamma_1(s_a)$ and $\gamma_2(s_b)$ on the different curves is,
\begin{eqnarray}
\|\gamma_1(s)-\gamma_2(s)\|_2
&=&\bigg \| \int_{0}^{s} \left ( C_1 \sin \int_{0}^{s} \kappa_1 (\sigma) d \sigma -C_2  \cos \int_{0}^{s} \kappa_1 (\sigma) d \sigma \right) d \sigma   \nonumber \\
&-&  \int_{0}^{s} \left ( C_1 \sin \int_{0}^{s} \kappa_2(\sigma) d \sigma -C_2 \cos \int_{0}^{s} \kappa_2 (\sigma) d \sigma \right) d \sigma \bigg \|_2 \nonumber \\
&=& \bigg \| \int_{0}^{s} \left ( C_1  \left (  \sin \int_{0}^{s} \kappa_1 (\sigma) d \sigma - \sin \int_{0}^{s} \kappa_2 (\sigma) d \sigma \right) \right) d \sigma  \nonumber \\
&+& \int_{0}^{s} \left ( C_2  \left (  \cos \int_{0}^{s} \kappa_2 (\sigma) d \sigma - \cos \int_{0}^{s} \kappa_1 (\sigma) d \sigma \right) \right) d \sigma  \bigg \|_2  \nonumber \\
&=& \bigg \| \int_{0}^{s} \left ( 2C_1 \cos \left ( \int_{0}^{s} \frac{\kappa_1 (\sigma) +\kappa_2(\sigma) }{2} d \sigma \right ) \sin \left ( \int_{0}^{s} \frac{\kappa_1 (\sigma)  -\kappa_2 (\sigma) }{2} d \sigma \right ) \right) d \sigma  \nonumber \\
&+& \int_{0}^{s} \left ( 2C_2 \sin \left ( \int_{0}^{s} \frac{\kappa_2 (\sigma) +\kappa_1 (\sigma) }{2} d \sigma \right ) \sin \left ( \int_{0}^{s} \frac{\kappa_1 (\sigma)  -\kappa_2 (\sigma) }{2} d \sigma \right ) \right) d \sigma \bigg \|_2  \nonumber \\
&=& \bigg \| \int_{0}^{s} 2 \left ( C_1 \cos \left ( \int_{0}^{s} \frac{\kappa_1 (\sigma) +\kappa_2 (\sigma) }{2} d \sigma \right ) + C_2 \sin \left ( \int_{0}^{s} \frac{\kappa_2 (\sigma) +\kappa_1 (\sigma) }{2} d \sigma \right )  \right) \nonumber \\
&\times& \sin \left ( \int_{0}^{s} \frac{\kappa_1 (\sigma)  -\kappa_2(\sigma) }{2} d \sigma \right ) d \sigma \bigg \|_2  \nonumber \\
&\le&  \int_{0}^{s} 2 \left ( (\|C_1\|_2+\|C_2\|_2) \bigg |  \sin \left ( \int_{0}^{s} \frac{\kappa_1 (\sigma)  -\kappa_2 (\sigma) }{2} d \sigma \right ) \bigg | \right) d \sigma
\end{eqnarray}
By the assumed condition $\sup_{s}|\kappa_1(s)-\kappa_2(s)| < \varepsilon$, we have at least for $\epsilon<\frac{\pi}{2}$,

\begin{eqnarray}\label{4thTrans}
\|\gamma_1(s)-\gamma_2(s)\|_2  
&<&  \int_{0}^{s} 2 \left ( (\|C_1\|_2+\|C_2\|_2) \bigg |  \sin \left ( \int_{0}^{s} \frac{\varepsilon}{2} d \sigma \right ) \bigg | \right) d \sigma \nonumber \\
&=&\int_{0}^{s} 2 \left ( \|  C_1 \|_2 + \|  C_2 \|_2 \right )  | \sin \left ( \frac{s \varepsilon}{2} \right)  |  d\sigma \nonumber \\
&=&  2s \left ( \|  C_1 \|_2 + \|  C_2 \|_2 \right )  |  \sin \left ( \frac{s \varepsilon}{2} \right)  | \leq s^2 \varepsilon ( \|  C_1 \|_2 + \|  C_2 \|_2 ) .
\end{eqnarray}
the last inequality following from the fact, $\sin(p)\leq p$ for all $0\leq p$.
Hence we may choose, $\delta( \varepsilon)=s^2 \varepsilon ( \|  C_1 \|_2 + \|  C_2 \|_2 )$.
Note that $s$ is the arc length from the beginning points to compared points, and recall that by assumptions that $\epsilon>0$. 
\end{proof}

\begin{theorem}\label{xx3.6}
(\ref{3.6})
For two closed curves $\gamma_1(s)=(x_1(s), y_1(s))$ and $\gamma_2(s)=(x_2(s), y_2(s)) (0\le s < 2\pi)$ which are boundaries of sets $A_1$ and $A_2$.   See Fig.~\ref{AreaIntersection}.  Let the boundaries of $A=A_1 \cap A_2$
$A=A_1 \cap A_2\neq \emptyset$ 
such that $area(A)>0$ 
be $\gamma(s)=(x(s),y(s))$, and,
\begin{eqnarray}\label{1stShapeCoherent}
\varepsilon= \max\{|x-x_2|,|y-y_2|,|x'-x_2'|,|y'-y_2'| \} \nonumber \\
M= 2\max\{ |x|,|x_2|,|y|,|y_2|,|x'|,|x_2'|,|y'|,|y_2'|\},
\end{eqnarray}
then there exist a $\Delta(\epsilon)$, which is defined as, 
\begin{eqnarray}\label{2stShapeCoherent}
\Delta(\epsilon)=\frac{2\pi M \varepsilon }{Area(A_2)}
\end{eqnarray}
such that 
\begin{eqnarray}\label{3stShapeCoherent}
1\geq \alpha(A_1, A_2, 0)\ge1-\Delta(\epsilon)
\end{eqnarray}
\end{theorem}


\begin{proof}
Suppose that the closed curve $\gamma_1(s)$ and $\gamma_2(s) $ are boundaries of sets $A_1$ and $A_2$, respectively.
Let $A=A_1 \cap A_2$ and $\gamma(s)$ be the boundary of $A$.
See Fig. \ref{AreaIntersection}. 
We annotate  these closed curves  as $\gamma(s)=(x(s),y(s))$,
$\gamma_1(s)=(x_1(s), y_1(s))$ and $\gamma_2(s)=(x_2(s), y_2(s))$.

By Green's theorem, we have,
\begin{eqnarray}\label{2ndShapeCoherent}
Area(A) &=& \frac{1}{2} \int_{\gamma(s)} xdy-ydx=\frac{1}{2} \int^{2\pi}_{0} x(s)y'(s)ds - y(s)x'(s)ds \nonumber \\
Area(A_2) &=& \frac{1}{2} \int_{\gamma_2(s)} x_2dy_2-y_2dx_2
=\frac{1}{2} \int^{2\pi}_{0} x_2(t)y_2'(s)ds-y_2(s)x_2'(s)ds \nonumber \\ 
&=&\frac{1}{2} \int^{2\pi}_{0} x_2(s)y_2'(s)ds - y_2(s)x_2'(s)ds
\end{eqnarray}
where the reason we use $2\pi$ is that we set the origin is inside $A$.
Then we have,
\begin{eqnarray}\label{3rdShapeCoherent}
|Area(A) - Area(A_2)| &=& \bigg | \frac{1}{2} \int^{2\pi}_{0} x(s)y'(s)ds - y(s)x'(s)ds
- \frac{1}{2} \int^{2\pi}_{0} x_2(s)y_2'(s)ds - y_2(s)x_2'(s)ds \bigg | \nonumber \\
&=& \bigg | \frac{1}{2} \int^{2\pi}_{0} \bigg  ( x(s)y'(s)-x_2(s)y_2'(s) \bigg  ) 
+\bigg  (y_2(s)x_2'(s)-y(s)x'(s) \bigg ) ds \bigg | \nonumber \\
&=& \bigg | \frac{1}{2} \int^{2\pi}_{0} \bigg  (x(s)y'(s)-x(s)y_2'(s)+x(s)y_2'(s)-x_2(s)y_2'(s) \bigg )   \nonumber \\
&+& \bigg  (y_2(s)x_2'(s)-y(s)x_2'(s)+y(s)x_2'(s)-y(s)x'(s) \bigg ) ds \bigg | \nonumber \\
&=& \bigg | \frac{1}{2} \int^{2\pi}_{0} \bigg  (x(s)y'(s)-x(s)y_2'(s) \bigg )+ \bigg (x(s)y_2'(s)-x_2(s)y_2'(s) \bigg  )   \nonumber \\
&+& \bigg  (y_2(s)x_2'(s)-y(s)x_2'(s) \bigg )+\bigg (y(s)x_2'(s)-y(s)x'(s) \bigg ) ds \bigg | \nonumber \\
&\le&  \frac{1}{2} \int^{2\pi}_{0} \bigg |x(s)y'(s)-x(s)y_2'(s) \bigg |+ \bigg |x(s)y_2'(s)-x_2(s)y_2'(s) \bigg |  \nonumber \\
&+& \bigg |y_2(s)x_2'(s)-y(s)x_2'(s) \bigg |+\bigg |y(s)x_2'(s)-y(s)x'(s) \bigg | ds  \nonumber \\
&\le&  \frac{1}{2} \int^{2\pi}_{0}  |x(s) | |y'(s)-y_2'(s) |+ |y_2'(s)| |x(s)-x_2(s)|  \nonumber \\
&+& |x_2'(s)| |y(s)-y_2(s)|+ |y(s)||x_2'(s)-x'(s) | ds  \nonumber \\ 
&\le&  \frac{1}{2} \int^{2\pi}_{0} \frac{1}{2} M\varepsilon + \frac{1}{2} M\varepsilon 
+ \frac{1}{2} M\varepsilon+\frac{1}{2} M\varepsilon ds  \nonumber \\ 
&=&  M\varepsilon \int^{2\pi}_{0}  ds  \nonumber \\ 
&=& 2\pi M\varepsilon 
\end{eqnarray}
On the other hand, from Eq.~(\ref{shapecoherenced}),
{\begin{eqnarray}
\displaystyle \alpha(A_1, A_2,0) = \sup_{S(A_2)} \frac{m( S(A_2) \cap \Phi_{0}(A_1)) }{m(A_2)}
\end{eqnarray}}
and then we have,
{\begin{eqnarray}
\displaystyle \alpha(A_1, A_2,0) \geq  \frac{m( S(A_2) \cap A_1) }{m(A_2)} 
=\frac{m(A)}{m(A_2)}= 1-\frac{|Area(A) - Area(A_2)|}{Area(A_2)}.
\end{eqnarray}}
Let $\Delta(\epsilon)=\frac{2\pi M\varepsilon }{Area(A_2)}$, we have 
{\begin{eqnarray}
\displaystyle \alpha(A_1, A_2,0) =1-\frac{|Area(A) - Area(A_2)|}{Area(A_2)} \ge 1-\Delta(\epsilon).
\end{eqnarray}}
\end{proof}

Note that the simplified details of the above proof assumes that there is a region of overlap between $A_1$ and $A_2$ that has just one connected  component, and we can always assume that the two regions are arranged to have at least one such region.  In the more general case that the two regions overlap in multiple components, then the above proof can be easily adjusted to integrate area across each region, and the statement of the theorem remains the same.  Notice also that we have stated the theorem so that each boundary curve has the same arc length, but in the case of different arc lengths a comparable theorem can be developed, proved by defining a ``speed" $s'=vs$, to show regularity between boundary curves leads to regularity of the shape coherence.

\nocite{*}
\bibliography{aipsamp}                    

\end{document}